\documentclass[a4paper,12pt]{article}
\usepackage{amssymb}
\usepackage{amsmath,amssymb,amsfonts}
\usepackage{graphicx}
\usepackage{epsfig}
\usepackage{multirow}
\usepackage{mathrsfs}
\usepackage{amsfonts}
\usepackage{amsmath}
\usepackage{dsfont}
\usepackage{graphicx}
\usepackage{subfigure}
\usepackage[colorinlistoftodos]{todonotes}
\usepackage[colorlinks=true, allcolors=blue]{hyperref}
\usepackage{amsthm}
\usepackage{float}
\usepackage{hyperref}
\usepackage{float}
\theoremstyle{plain}
\usepackage[english]{babel}
\usepackage[utf8x]{inputenc}
\usepackage[T1]{fontenc}
\newtheorem{theorem}{Theorem}[section]
\newtheorem{lemma}{Lemma}[section]

\newtheorem{remark}[theorem]{Remark}

\numberwithin{equation}{section}
\numberwithin{lemma}{section}

\usepackage{bm}

\title{Extended sampling method for inverse elastic scattering problems using one incident wave}
\author{J. Liu
\thanks{Department of Mathematical Sciences, Jinan University, Guangzhou, 130012, China ({\tt liujuan@jnu.edu.cn}).}
\and X. Liu
\thanks{NCMIS and Academy of Mathematics and Systems Science,
Chinese Academy of Sciences, Beijing 100190, China. ({\tt xdliu@amt.ac.cn}).}
\and J. Sun
\thanks{Department of Mathematical Sciences, Michigan Technological University, Houghton, MI 49931, U.S.A. ({\tt jiguangs@mtu.edu}).}
}
\date{}
\begin{document}
\maketitle
\begin{abstract}
We consider the inverse elastic scattering problems using the far field data due to one incident plane wave. A simple method is proposed to reconstruct the location and size of the obstacle using different components of the far field pattern. The method sets up linear ill-posed integral equations for sampling points in the domain of interrogation and uses the (approximate) solutions to compute indicators. Using the far field patterns of rigid disks as the kernels of the integral equations and moving the measured data to the right hand side, the method has the ability to process limited aperture data. Numerical examples show that the method can effectively determine the location and approximate the support of the obstacle with little a priori information.
\end{abstract}
\section{Introduction}
Motivated by applications in non-destructive testings, medical imaging and seismic exploration,
the inverse elastic scattering problems have received significant attention recently
\cite{HahnerHsiao1993IP, Arens2001, BonnetConstantinescu2005IP, AmmariCalmonIakovleva2008SIAMIS, HuLiLiuSun2014SIAMIS, LiWangWangZhao2016IP, JiLiuXi2018, BaoHuSunYin2018JMPA}. There exist two main groups of methods to recover the location and shape of an
obstacle using the scattering data.
The first group are the iterative methods to minimize some cost functions
\cite{BaoHuSunYin2018JMPA, LiWangWangZhao2016IP}.
These methods usually need to solve the forward scattering problems.
The second group are the non-iterative or direct methods, e.g., the linear sampling method \cite{ColtonKirsch1996IP, Arens2001}, the factorization method \cite{AlvesKress2002,Charalambopoulos2007,HuKirschSini2013}, the reciprocity gap method \cite{ColtonHaddar2005IP, MonkSun2007IPI, DiCristoSun2006IP},
the range test method
\cite{PotthastSylvestrKusiak2003IP}, the reverse time migration \cite{ChenHuang2015SCM} and the direct sampling method \cite{ItoJinZou2012IP, JiLiuXi2018}.
These methods do not solve any forward scattering problem and thus are fast in general.

In this paper, we consider the inverse elastic scattering problems using the far field pattern due to one incident plane wave.
A simple method, called extended sampling method (ESM) \cite{LiuSun2018}, is proposed to reconstruct the location and size of the obstacle.
The method is based on an idea similar to that of the linear sampling method,
which has been studied extensively in the literature \cite{ColtonKirsch1996IP, ColtonKress2013}.

The linear sampling method uses the full aperture far field pattern, i.e., the far field pattern of all incident and observation directions.
It sets up linear ill-posed integral equations for the sampling points in the domain of interrogation and uses the (approximate)
solutions to reconstruct the location and shape of the unknown obstacle.
The kernel of the integral equations is the (measured) full aperture far field pattern.
ESM also sets up integral equations for the sampling points in the domain of interrogation
and uses the (approximate) solutions to compute some indicators.
However, the kernels of these integral equations for ESM are the full aperture far field patterns of rigid disks, which can be computed easily in advance.
The measured far field pattern is moved to the right hand side of the equation.
This arrangement gives ESM the ability to process the far field pattern due to a single incident plane wave.

The rest of the paper is organized as follows.
In Section  \ref{DSP}, the direct elastic scattering problem is presented.
In Section \ref{section2}, we first derive a relation between the scattered field of the elastic scattering problems and the Sommerfeld radiation solutions of Helmholtz equations. This relation is used to transform the inverse problem with the compressional part or shear part of the far field pattern into an inverse acoustic scattering problem. Then a reconstruction algorithm is proposed based on ESM. In Section \ref{section3}, using the far field patterns of elastic waves for rigid disks and the related translation property, a new far field equation is proposed. The behavior of the solutions is analyzed. An ESM algorithm for the inverse problem using the new equation is presented. In Section \ref{section4}, numerical examples are provided.

\section{Preliminaries and the direct scattering problem}\label{DSP}
We begin with the notations used throughout this paper. Vectors are written in bold to distinguish from scalars. For a vector ${\boldsymbol x}=(x_1; x_2):=(x_1, x_2)^{T}\in \mathbb{R}^2$, let $\hat{{\boldsymbol x}}:={\boldsymbol x}/|{\boldsymbol x}|$ and $\hat{\boldsymbol x}^{\perp}$ be obtained by rotating ${\hat{\boldsymbol x}}$ $\pi/2$ anticlockwise. Denote $\partial/\partial x_i, i=1,2$ by $\partial_i$ for simplicity. In addition to the usual differential operators $\textrm{grad}{ u}:=({\partial_{1} u}; {\partial_{2} u})$ and $\textrm{div}{\boldsymbol u}:={\partial_{1} u_1}+{\partial_{2} u_2}$, we will make use of $\textrm{grad}^{\perp}{u}:=(-{\partial_{2} u}; {\partial_{1} u})$ and $\textrm{div}^{\perp}{\boldsymbol u}:={\partial_{1} u_2}-{\partial_{2} u_1}$.

The propagation of time-harmonic waves in an isotropic homogeneous medium with Lame constants $\lambda$, $\mu$ ($\mu>0$, $2\mu+\lambda>0$) and density $\rho$ is governed by the Navier equation
\begin{equation}\label{navier}
\Delta^*  {\boldsymbol u}+\rho \omega^2  {\boldsymbol u}={\boldsymbol 0},
\end{equation}
where ${\boldsymbol u}$ denotes the displacement field and $\omega$ denotes the circular frequency. The differential operator $\Delta^*{\boldsymbol u}:=\mu \Delta {\boldsymbol u}+(\lambda+\mu) \textrm{ grad div}\, {\boldsymbol u}$. In this paper, we assume $\rho\equiv 1$ for simplicity.

The solution ${\boldsymbol u}$ of \eqref{navier} can be decomposed as
\begin{equation*}
 {\boldsymbol u}= {\boldsymbol u}_p+ {\boldsymbol u}_s,
\end{equation*}
where
\begin{equation}\label{divide}
 {\boldsymbol u}_p:=-\frac{1}{k_p^2}\textrm{ grad div } {\boldsymbol u},\quad  {\boldsymbol u}_s:=-\frac{1}{k_s^2}\textrm{ grad}^\perp \textrm{div}^\perp {\boldsymbol u}
\end{equation}
are known as the compressional part of ${\boldsymbol u}$ associated with the wave number $k_p:={\omega}/{\sqrt{2\mu+\lambda}}$, and the shear part of $ {\boldsymbol u}$ associated with the wave number $k_s:={\omega}/{\sqrt{\mu}}$.

{\it The direct elastic scattering problem} for an obstacle is as follows. Given a bounded domain $D$ of class $C^2$ and an incident field ${\boldsymbol u}^{\textrm{inc}}$ such that ${\boldsymbol u}^{\textrm{inc}}|_{\partial D}\in [C(\partial D)]^2$ and ${\boldsymbol u}^{\textrm{inc}}$ is a solution of \eqref{navier} in a neighborhood of $\partial D$, find the scattered field ${\boldsymbol u}\in [C^2(\mathbb{R}^2\setminus \overline{D})\bigcap C^1(\mathbb{R}^2\setminus D)]^2$ such that
\begin{equation}\label{obstacle}
  \left\{
   \begin{array}{lll}
   &\Delta^* {\boldsymbol u}+\omega^2 {\boldsymbol u}={\boldsymbol 0},\ \ \ \textrm{in} \ \mathbb{R}^2\setminus \overline{D},\\
   &\lim\limits_{{r}\rightarrow\infty}\sqrt{r}(\partial {\boldsymbol u}_p/\partial r-{i} k_p {\boldsymbol u}_p)={\boldsymbol 0},\ \ \ r=| {\boldsymbol x}|,\\
   &\lim\limits_{{r}\rightarrow\infty}\sqrt{r}(\partial {\boldsymbol u}_s/\partial r-{i} k_s {\boldsymbol u}_s)={\boldsymbol 0},\ \ \ r=| {\boldsymbol x}|.
   \end{array}\right.
\end{equation}
It is well known that every radiation solution to the Navier equation has an asymptotic behavior of the form
\begin{equation}\label{farfield}
 {\boldsymbol u}({\boldsymbol x})=\frac{e^{ik_p|{\boldsymbol x}|}}{\sqrt{|{\boldsymbol x}|}}{u}_p^\infty({\hat{\boldsymbol x}}){\hat{\boldsymbol x}}+\frac{e^{ik_s|{\boldsymbol x}|}}{\sqrt{|{\boldsymbol x}|}}{u}_s^\infty({\hat{\boldsymbol x}}){\hat{\boldsymbol x}}^\perp+O\big(|{\boldsymbol x}|^{-3/2}\big), \ \ \ | {\boldsymbol x}|\rightarrow \infty,
\end{equation}
uniformly in all directions ${\hat{\boldsymbol x}}$, where ${u}_p^\infty$ and ${u}_s^\infty$ are analytic functions on the unit circle 
\[
\mathbb{S}:=\{{\hat{\boldsymbol x}} | \hat{\boldsymbol x} \in \mathbb{R}^2, |\hat{\boldsymbol x}|=1\}.
\] 
Throughout the paper, ${\boldsymbol u}_\infty(\hat{\boldsymbol x}):=({u}_{p}^\infty(\hat{\boldsymbol x});{u}_s^\infty(\hat{\boldsymbol x}))$ is defined as the far field pattern of ${\boldsymbol u}$, and ${u}_p^\infty$ and ${u}_s^\infty$ are defined as compressional part and shear part of the far-field pattern, respectively.
For the well-posedness of the above direct scattering problem, one
needs to impose suitable conditions on $\partial D$, which depend on the physical properties of the scatterer. The scattered field $ {\boldsymbol u}$ satisfies
\begin{itemize}
\item[1)] the Dirichlet boundary condition
\begin{equation*}\label{dirichlet}
 {\boldsymbol u}=- {\boldsymbol u}^{\textrm{inc}} \ \ \ \textrm{on}\  \partial D,
\end{equation*}
for a rigid body;
\item[2)] the Neumann boundary condition
\begin{equation*}\label{neumann}
 {T_{\boldsymbol \nu}} {\boldsymbol u}=- {T_{\boldsymbol \nu}} {\boldsymbol u}^{\textrm{inc}} \ \ \ \textrm{on}\  \partial D,
\end{equation*}
for a cavity, where $ {T_{\boldsymbol \nu}}:=2\mu \frac{\partial }{\partial {\boldsymbol \nu}}+\lambda{\boldsymbol \nu} \textrm{ div }-\mu {\boldsymbol \nu}^{\perp}\textrm{div}^{\perp}$ denotes the surface traction operator and ${\boldsymbol \nu}$ is the unit outward normal to $\partial D$;
\item[3)] the impedance boundary condition
\begin{equation*}\label{impedance}
 {T_{\boldsymbol \nu}} {\boldsymbol u} +{i}\sigma  {\boldsymbol u}=- {T_{\boldsymbol \nu}} {\boldsymbol u}^{\textrm{inc}} -{i}\sigma  {\boldsymbol u}^{\textrm{inc}}\ \ \ \textrm{on}\  \partial D,
\end{equation*}
with some real-valued parameter $\sigma\geq 0$.
\end{itemize}

{\it The inverse scattering problem} of interests is, using the far-field pattern of all observation directions due to one incident wave, to reconstruct the location and approximate the support of the scatterer without knowing the physical properties of the scatterer. More specifically, the following two inverse elastic scattering problems will be considered:
\begin{itemize}

 \item[{\bf IP-P:}] {\em Determine the location and size of the scatterer $D$
 from the knowledge of compressional part ${u}_p^\infty(\hat{\boldsymbol x}), \hat{\boldsymbol x}\in \mathbb{S}$ or shear part ${u}_s^\infty(\hat{\boldsymbol x}), \hat{\boldsymbol x}\in \mathbb{S}$ of the far field pattern due to one incident wave.}

\item[{\bf IP-F:}] {\em Determine the location and size of the scatterer  $D$
 from the knowledge of the far field pattern ${\boldsymbol u}_\infty(\hat{\boldsymbol x})=({u}_{p}^\infty(\hat{\boldsymbol x});{u}_s^\infty(\hat{\boldsymbol x})), \hat{\boldsymbol x}\in \mathbb{S}$ due to one incident wave.}
\end{itemize}

\section{Extended sampling method for IP-P}\label{section2}
By building a relation between the scattered solution of elastic scattering problems and the radiating solutions of Helmholtz equations, we can extend ESM for inverse acoustic scattering problems to solve {\bf IP-P}.
\subsection{Radiating solutions of Helmholtz equations}
Recall that ${\boldsymbol u}$ is the scattered field of \eqref{obstacle} and ${\boldsymbol u}_\infty(\hat{\boldsymbol x}):=({u}_{p}^\infty(\hat{\boldsymbol x});u_s^\infty(\hat{\boldsymbol x}))$ is the corresponding far field pattern. The following theorem shows that ${\boldsymbol u}$ and ${\boldsymbol u}_\infty$ are related with the radiating solutions and their far field patterns of some Helmholtz equations, respectively.\\

\begin{theorem}\label{theorem1}
Let ${\boldsymbol u}$ be a solution of \eqref{obstacle}. Then $\phi:=-\frac{1}{k_p^2}{\rm div}\,{\boldsymbol u}$ and $\psi:=-\frac{1}{k_s^2}{\rm div}^\perp \,{\boldsymbol u}$ are the radiating solutions of
\begin{equation}\label{obstacle2}
  \left\{
   \begin{array}{lll}
   &\Delta  {\phi}+k^2_p  {\phi}=0,\ \ \ \textrm{in} \ \mathbb{R}^2\setminus \overline{D},\\
   &\Delta  {\psi}+k^2_s  {\psi}=0,\ \ \ \textrm{in} \ \mathbb{R}^2\setminus \overline{D},\\
   &\lim\limits_{{r}\rightarrow\infty}\sqrt{r}(\partial  \phi/\partial r-{i} k_p  \phi)=0,\ \ \ \ r=| {\boldsymbol x}|,\\
   &\lim\limits_{{r}\rightarrow\infty}\sqrt{r}(\partial  \psi/\partial r-{i} k_s  \psi)=0,\ \ \ r=| {\boldsymbol x}|.
   \end{array}\right.
\end{equation}
Furthermore, for the far field patterns $\phi_\infty$ and $\psi_\infty$ of $\phi$ and $\psi$, respectively,
\begin{equation}\label{farrelation_h_e}
\phi_\infty(\hat{\boldsymbol x})=\frac{1}{ik_p}{u}_p^\infty(\hat{\boldsymbol x}),\ \ \ \ \psi_\infty(\hat{\boldsymbol x})=\frac{1}{ik_s}{u}_s^\infty(\hat{\boldsymbol x}).
\end{equation}
\end{theorem}

\begin{proof}
It can be easily verified that $\phi=-\frac{1}{k_p^2}{\textrm{div}}\ {\boldsymbol u}$ and $\psi=-\frac{1}{k_s^2}{\textrm{div}}^\perp {\boldsymbol u}$ satisfy the first and second Helmholtz equations of \eqref{obstacle2}, respectively.

From \eqref{farfield}, the radiating solution of \eqref{obstacle} in polar coordinates has an asymptotic behavior
\begin{eqnarray*}
{\boldsymbol u}(r,\theta)&=&{\left( \begin{array}{c}
u_1(r,\theta) \\
u_2(r,\theta)
\end{array}
\right )}\\
&=&\frac{e^{ik_pr}}{\sqrt{r}}u_p^\infty(\theta){\left( \begin{array}{c}
\cos\theta \\
\sin\theta
\end{array}
\right )}+\frac{e^{ik_sr}}{\sqrt{r}}u_s^\infty(\theta){\left( \begin{array}{c}
-\sin\theta \\
\cos\theta
\end{array}
\right )}+O(r^{-3/2}).
\end{eqnarray*}
Consequently, we have that
\begin{equation*}
\frac{\partial}{\partial r}u_1(r,\theta)=ik_p\frac{e^{ik_pr}}{\sqrt{r}}u_p^\infty(\theta)\cos\theta-ik_s\frac{e^{ik_sr}}{\sqrt{r}}u_s^\infty(\theta)\sin\theta+O(r^{-3/2}),
\end{equation*}
\begin{equation*}
\frac{\partial}{\partial r}u_2(r,\theta)=ik_p\frac{e^{ik_pr}}{\sqrt{r}}u_p^\infty(\theta)\sin\theta+ik_s\frac{e^{ik_sr}}{\sqrt{r}}u_s^\infty(\theta)\cos\theta+O(r^{-3/2}).
\end{equation*}
Using these two equations, we obtain
\begin{eqnarray*}
\textrm{div}\ {\boldsymbol u}&=&\frac{\partial u_1}{\partial x_1}+\frac{\partial u_2}{\partial x_2}\\
&=&\bigg(\frac{\partial u_1}{\partial r}\cos\theta-\frac{1}{r}\frac{\partial u_1}{\partial \theta}\sin\theta\bigg)+\bigg(\frac{\partial u_2}{\partial r}\sin\theta+\frac{1}{r}\frac{\partial u_2}{\partial \theta}\cos\theta\bigg)\\
&=&\bigg(\frac{\partial u_1}{\partial r}\cos\theta+\frac{\partial u_2}{\partial r}\sin\theta\bigg)+\bigg(\frac{1}{r}\frac{\partial u_2}{\partial \theta}\cos\theta-\frac{1}{r}\frac{\partial u_1}{\partial \theta}\sin\theta\bigg)\\
&=&ik_p\frac{e^{ik_pr}}{\sqrt{r}}u_p^\infty(\theta)+O(r^{-3/2})
\end{eqnarray*}
and
\begin{eqnarray*}
\textrm{div}^\perp\ {\boldsymbol u}&=&\frac{\partial u_2}{\partial x_1}-\frac{\partial u_1}{\partial x_2}\\
&=&\bigg(\frac{\partial u_2}{\partial r}\cos\theta-\frac{1}{r}\frac{\partial u_2}{\partial \theta}\sin\theta\bigg)-\bigg(\frac{\partial u_1}{\partial r}\sin\theta+\frac{1}{r}\frac{\partial u_1}{\partial \theta}\cos\theta\bigg)\\
&=&\bigg(\frac{\partial u_2}{\partial r}\cos\theta-\frac{\partial u_1}{\partial r}\sin\theta\bigg)-\bigg(\frac{1}{r}\frac{\partial u_1}{\partial \theta}\cos\theta+\frac{1}{r}\frac{\partial u_2}{\partial \theta}\sin\theta\bigg)\\
&=&ik_s\frac{e^{ik_sr}}{\sqrt{r}}u_s^\infty(\theta)+O(r^{-3/2}).
\end{eqnarray*}
Thus
\begin{equation}\label{phi_asymptitic}
\phi({\boldsymbol x})=-\frac{1}{k_p^2}{\textrm{div}}\ {\boldsymbol u}=\frac{1}{ik_p}\frac{e^{ik_p|{\boldsymbol x}|}}{\sqrt{|{\boldsymbol x}|}}u_p^\infty(\hat{\boldsymbol x})+O(|{\boldsymbol x}|^{-3/2}),
\end{equation}
\begin{equation}\label{psi_asymptitic}
\psi({\boldsymbol x})=-\frac{1}{k_s^2}{\textrm{div}}^\perp {\boldsymbol u}=\frac{1}{ik_s}\frac{e^{ik_s|{\boldsymbol x}|}}{\sqrt{|x|}}u_s^\infty(\hat{\boldsymbol x})+O(|{\boldsymbol x}|^{-3/2}),
\end{equation}
which imply that $\phi$ and $\psi$ satisfy the Sommerfeld radiation conditions in \eqref{obstacle2}.

Since  $\phi=-\frac{1}{k_p^2}{\textrm{div}}\ {\boldsymbol u}$ and $\psi=-\frac{1}{k_s^2}{\textrm{div}}^\perp {\boldsymbol u}$ are radiating solutions of the Helmholtz equations, the far field patterns $\phi_\infty(\hat{\boldsymbol x})$ and $\psi_\infty(\hat{\boldsymbol x})$ have the following asymptotic expansions (see \cite{ColtonKress2013})
\begin{equation*}
\phi({\boldsymbol x})=\frac{e^{ik_p|{\boldsymbol x}|}}{\sqrt{|{\boldsymbol x}|}}\bigg\{\phi_\infty(\hat{\boldsymbol x})+O\bigg(\frac{1}{|{\boldsymbol x}|}\bigg)\bigg\},\ \ \ |{\boldsymbol x}|\rightarrow \infty,
\end{equation*}
\begin{equation*}
\psi({\boldsymbol x})=\frac{e^{ik_s|{\boldsymbol x}|}}{\sqrt{|{\boldsymbol x}|}}\bigg\{\psi_\infty(\hat{\boldsymbol x})+O\bigg(\frac{1}{|{\boldsymbol x}|}\bigg)\bigg\},\ \ \ |{\boldsymbol x}|\rightarrow \infty.
\end{equation*}
Then from \eqref{phi_asymptitic} and \eqref{psi_asymptitic}, \eqref{farrelation_h_e} is proved.
\end{proof}

\begin{remark}\label{remark1}
If the obstacle $D$ is a rigid body, the solution of the elastic scattering problem \eqref{obstacle} satisfies the Dirichlet boundary condition ${\boldsymbol u}=-{\boldsymbol u}^{\rm{inc}}$. From \eqref{divide}, $\phi=-\frac{1}{k_p^2}{\rm div}\ {\boldsymbol u}$ and $\psi=-\frac{1}{k_s^2}{\rm div}^\perp {\boldsymbol u}$ satisfy
\begin{equation}
{\boldsymbol u}={\rm grad}\ \phi+{\rm grad}^\perp \psi.
\end{equation}
Denote by ${\boldsymbol \tau}=(\tau_1;\tau_2)$ the unit tangent vector and by ${\boldsymbol \nu}=(\nu_1;\nu_2)=(\tau_2; -\tau_1)$ the unit outward normal vector on $\partial D$. By straightforward calculation, $\phi$ and $\psi$ satisfy the coupled boundary conditions
\begin{equation*}
 \frac{\partial \phi}{\partial {\boldsymbol \nu}}+\frac{\partial\psi}{\partial {\boldsymbol \tau}}=-{\boldsymbol \nu}\cdot {\boldsymbol u}^{\rm{inc}},\ \ \
 \frac{\partial\phi}{\partial {\boldsymbol \tau}}-\frac{\partial\psi}{\partial {\boldsymbol \nu}}=-{\boldsymbol \tau}\cdot {\boldsymbol u}^{\rm{inc}},\ \ \ \textrm{on}\ \partial D.
\end{equation*}
From \autoref{theorem1}, $(\phi, \psi)$ satisfies the following Helmholtz equations with coupled boundary conditions
\begin{equation}\label{rigidbody1}
  \left\{
   \begin{array}{lll}
   &\Delta  {\phi}+k^2_p  {\phi}=0,\ \ \ \textrm{in} \ \mathbb{R}^2\setminus \overline{D},\\
   &\Delta  {\psi}+k^2_s  {\psi}=0,\ \ \ \textrm{in} \ \mathbb{R}^2\setminus \overline{D},\\
   &\frac{\partial {\boldsymbol \phi}}{\partial {\boldsymbol \nu}}+\frac{\partial\psi}{\partial\tau}=-{\boldsymbol \nu}\cdot {\boldsymbol u}^{\textrm{inc}},\ \ \ \textrm{on}\ \partial D,\\
   &\frac{\partial {\boldsymbol \phi}}{\partial {\boldsymbol \tau}}-\frac{\partial\psi}{\partial {\boldsymbol \nu}}=-{\boldsymbol \tau}\cdot {\boldsymbol u}^{\textrm{inc}},\ \ \ \textrm{on}\ \partial D,\\
   &\lim\limits_{{r}\rightarrow\infty}\sqrt{r}(\partial  \phi/\partial r-{i} k_p  \phi)=0,\ \ \ \ r=| {\boldsymbol x}|,\\
   &\lim\limits_{{r}\rightarrow\infty}\sqrt{r}(\partial  \psi/\partial r-{i} k_s  \psi)=0,\ \ \ r=| {\boldsymbol x}|,
   \end{array}\right.
\end{equation}
which has a unique solution (see \cite{WangQiu2017}). On the other hand, from \eqref{divide}, if $(\phi, \psi)$ solves \eqref{rigidbody1}, direct computation shows that ${\boldsymbol u}={\rm grad} \ \phi+{\rm grad}^\perp \psi$ satisfies the elastic scattering problem \eqref{obstacle} with Dirichlet boundary condition.
\end{remark}

\subsection{Extended sampling method} \label{2.2}

Denote by ${\boldsymbol U}({\boldsymbol x}), {\boldsymbol x}\in \mathbb{R}^2\setminus \overline{D}$ and ${\boldsymbol U}_\infty(\hat{\boldsymbol x}):=(U_p^\infty;U_s^\infty), \hat{\boldsymbol x}\in \mathbb{S}$ the scattered field and far field pattern of the unknown scatterer $D$ due to an incident wave, respectively. The inverse problem {\bf IP-P} is to determine the location and approximate support of the scatterer from the knowledge of the compressional part $U_p^\infty(\hat{\boldsymbol x})$ or the shear part $U_s^\infty(\hat{\boldsymbol x})$. In this subsection, for simplicity we use $t$ to represent either $p$ or $s$.

 Let $B_{\boldsymbol z}$ be a sound soft disk centered at ${\boldsymbol z}$ with radius $R$ and let $u^{B_{\boldsymbol z}}({\boldsymbol x};k_t,{\boldsymbol d})$, ${\boldsymbol x}\in \mathbb{R}^2\setminus \overline{B_{\boldsymbol z}}$ solve
\begin{equation*}
  \left\{
   \begin{array}{lll}
   &\Delta u+k_t^2 u=0,\ \ \ \textrm{in}\  \mathbb{R}^2\setminus \overline{B_{\boldsymbol z}}, \\
   &u=-{e}^{{i}k_t{\boldsymbol x}\cdot {\boldsymbol d}},\ \ \ \textrm{on}\ \partial B_{\boldsymbol z}\\
   &\lim\limits_{{r}\rightarrow\infty}\sqrt{r}(\partial u/\partial r-{i}k_tu)=0,\ \ \ r=|{\boldsymbol x}|,
   \end{array}\right.
\end{equation*}
where ${\boldsymbol d}\in \mathbb{S}$ is the incident direction. Denote $u_{\infty}^{B_{\boldsymbol z}}(\hat{\boldsymbol x};k_t,{\boldsymbol d})$ the far field pattern of $u^{B_{\boldsymbol z}}({\boldsymbol x};k_t,{\boldsymbol d})$ (for its series expansion, see e.g., \cite{ColtonKress2013, LiuSun2018}).
Then for {\bf IP-P}, we introduce the far field equation
\begin{equation}\label{fe}
\int_{\mathbb{S}}u^{B_{\boldsymbol z}}_{\infty}(\hat{\boldsymbol x};k_t,{\boldsymbol d})g_{\boldsymbol z}({\boldsymbol d})d s({\boldsymbol d})=\frac{1}{ik_t}U_t^\infty(\hat{\boldsymbol x}), \quad \hat{\boldsymbol x} \in \mathbb{S}.
\end{equation}

The following theorem is the main result for \eqref{fe}.\\

\begin{theorem}\label{theorem2}
Let $B_{\boldsymbol z}$ be a sound soft disk centered at ${\boldsymbol z}$ with radius $R$. Let $D$ be an obstacle. Assume that $k_tR$ does not coincide with any zero of the Bessel functions $J_n, n=0,1,2,\cdots$. Then the following results hold for the far field equation (\ref{fe}):
\begin{itemize}
\item[1.] If $D\subset B_{\boldsymbol z}$, for a given $\varepsilon>0$, there exists a function $g_{\boldsymbol z}^\varepsilon\in L^2(\mathbb{S})$ such that
\begin{equation}\label{fe2}
\bigg\|\int_{\mathbb{S}}u^{B_{\boldsymbol z}}_{\infty}(\hat{\boldsymbol x}; k_t, {\boldsymbol d})g_{\boldsymbol z}^\varepsilon({\boldsymbol d}) d s({\boldsymbol d})-\frac{1}{ik_t}U_t^\infty(\hat{\boldsymbol x})\bigg\|_{L^2(\mathbb{S})}<\varepsilon
\end{equation}
and the Herglotz wave function 
\[
v_{g_{\boldsymbol z}^\varepsilon}({\boldsymbol x}):=\int_{\partial B_{\boldsymbol z}}e^{ik_t {\boldsymbol x}\cdot {\boldsymbol d}}g_{\boldsymbol z}^\varepsilon ({\boldsymbol d})ds({\boldsymbol d}), {\boldsymbol x}\in B_{\boldsymbol z}
\] 
converges to the solution $w\in H^1(B_{\boldsymbol z})$ of the Helmholtz equation with $w=-\phi_t$ on $\partial B_{\boldsymbol z}$ as $\varepsilon\rightarrow 0$.

\item[2.] If $D\cap B_{\boldsymbol z}=\emptyset$, every $g_{\boldsymbol z}^\varepsilon\in L^2(\mathbb{S})$ that satisfies (\ref{fe2}) for a given $\varepsilon>0$ is such that
\begin{equation*}
\lim_{\varepsilon\rightarrow 0}\|v_{g_{\boldsymbol z}^\varepsilon}\|_{H^1(B_{\boldsymbol z})}=\infty.
\end{equation*}

\end{itemize}
\end{theorem}

\begin{proof}
Denote the right hand side of the far field equation \eqref{fe} by  $\phi_t^\infty(\hat{\boldsymbol x}):=\frac{1}{ik_t}U_t^\infty(\hat{\boldsymbol x})$. From \autoref{theorem1}, we know that $\phi_t^\infty(\hat{\boldsymbol x})$ is the far field pattern of
\begin{equation*}
  \phi_t({\boldsymbol x}):=\left\{
   \begin{array}{ll}
   &-\frac{1}{k_p^2}{\textrm{div}}\ {\boldsymbol U}, \ \ \ t=p,\\
   &-\frac{1}{k_s^2}{\textrm{div}}^\perp {\boldsymbol U},\ \ \ t=s,
   \end{array}\right.
\end{equation*}
which is a radiating solution of the Helmholtz equation outside $D$. The rest of proof is exactly the same as the that of Theorem 3.1 in \cite{LiuSun2018}.
\end{proof}

Now we are ready to present the extended sampling method (ESM) for {\bf IP-P}. Let $\Omega$ be a domain containing $D$. For a sampling point ${\boldsymbol z} \in \Omega$, we consider the linear ill-posed integral equation \eqref{fe}.
By \autoref{theorem2}, one expects that $\|g^\alpha_{\boldsymbol z}\|_{L^2(\mathbb{S})}$ is relatively large when $D$ is outside $B_{\boldsymbol z}$
and relatively small when $D$ is inside $B_{\boldsymbol z}$.
Consequently, a reconstruction of the location and support of $D$ can be obtained by plotting $\|g^\alpha_{\boldsymbol z}\|_{L^2(\mathbb{S})}$ for
all sampling points ${\boldsymbol z} \in \Omega$.\\

\begin{itemize}
\item[ ] \hspace{-1cm}{\bf{The Extended Sampling Method for IP-P}}{\em
\item[1.] Generate a set $T$ of sampling points for $\Omega$ which contains $D$.
\item[2.] For each sampling point ${\boldsymbol z} \in T$,
	\begin{itemize}
	\item[a.] compute $u^{B_{\boldsymbol z}}_{\infty}(\hat{\boldsymbol x}; k_t,{\boldsymbol d})$ and set up a discrete version of \eqref{fe};
	\item[b.] use the Tikhonov regularization to compute an approximate solution $g^\alpha_{\boldsymbol z}$ of \eqref{fe}.
	\end{itemize}
\item[3.] Find the global minimum point ${\boldsymbol z}^*\in T$ for $\|g^\alpha_{\boldsymbol z}\|_{L^2(\mathbb{S})}$.
\item[4.] Choose $B_{{\boldsymbol z}^*}$ to be the reconstruction for $D$.}\\
\end{itemize}

As in \cite{LiuSun2018}, one can use a multilevel technique to find a suitable radius $R$ of the sampling disks $B_{\boldsymbol z}, {\boldsymbol z}\in T$.\\

\begin{itemize}
\item[ ] \hspace{-1cm}{\bf{The Multilevel ESM for IP-P}}{\em
\item[1.] Choose the sampling disks with a large radius $R_0$. Generate a proper set $T_0$ of sampling points.
Using ESM, determine the global minimum point ${\boldsymbol z}_0\in T_0$ for  $\|g^\alpha_{\boldsymbol z}\|_{L^2(\mathbb{S})}$ and an approximation $D_0$ for $D$.

\item[2.] For $j=1,2,\cdots$
	\begin{itemize}
	\item Let $R_j=R_{j-1}/2$ and generate a proper set $T_j$ of sampling points.
	\item Find the minimum point ${\boldsymbol z}_j\in T_j$  for  $\|g^\alpha_{\boldsymbol z}\|_{L^2(\mathbb{S})}$ and an approximation $D_j$ for $D$.
		 If ${\boldsymbol z}_{j} \not\in D_{j-1}$, go to Step 3.
	\end{itemize}

\item[3.] Choose ${\boldsymbol z}_{j-1}$ and $D_{j-1}$ to be the location and approximate support of $D$, respectively.}
\end{itemize}

\section{Extended sampling method for IP-F}\label{section3}
In this section, a novel far field equation is introduced.
The series expansion for the far field patterns of rigid disks, which serve as the kernels of the integrals, will be studied in detail.
Then an extended sampling method for {\bf IP-F} is proposed.
\subsection{Far field pattern for rigid disks}\label{appendixA}
Let $B\subset \mathbb{R}^2$ be a rigid disk centered at the origin with radius $R$. Let ${\boldsymbol u}^B({\boldsymbol x})$ and ${\boldsymbol u}^B_\infty=(u_p^{\infty,{\boldsymbol 0}};u_s^{\infty,{\boldsymbol 0}})$ be the radiating solution and far field pattern of the elastic scattering problem \eqref{obstacle} for $B$ due to an incident wave ${\boldsymbol u}^{\textrm{inc}}({\boldsymbol x})$, respectively. From Remark \ref{remark1}, we know that ${\boldsymbol u}^B({\boldsymbol x})=\textrm{grad}\ \phi+\textrm{grad}^\perp \psi$, where $(\phi, \psi)$ is the unique solution of
\begin{equation*}\label{rigidbody}
  \left\{
   \begin{array}{lll}
   &\Delta  {\phi}+k^2_p  {\phi}=0,\ \ \ \textrm{in} \ \mathbb{R}^2\setminus \overline{B},\\
   &\Delta  {\psi}+k^2_s  {\psi}=0,\ \ \ \textrm{in} \ \mathbb{R}^2\setminus \overline{B},\\
   &\frac{\partial \phi}{\partial{\boldsymbol \nu}}+\frac{\partial\psi}{\partial{\boldsymbol \tau}}=-{\boldsymbol \nu}\cdot {\boldsymbol u}^{\textrm{inc}},\ \ \ \textrm{on}\ \partial B,\\
   &\frac{\partial\phi}{\partial{\boldsymbol \tau}}-\frac{\partial\psi}{\partial{\boldsymbol \nu}}=-{\boldsymbol \tau}\cdot {\boldsymbol u}^{\textrm{inc}},\ \ \ \textrm{on}\ \partial {B},\\
   &\lim\limits_{{r}\rightarrow\infty}\sqrt{r}(\partial  \phi/\partial r-{i} k_p  \phi)=0,\ \ \ \ r=| {\boldsymbol x}|,\\
   &\lim\limits_{{r}\rightarrow\infty}\sqrt{r}(\partial  \psi/\partial r-{i} k_s  \psi)=0,\ \ \ r=| {\boldsymbol x}|.
   \end{array}\right.
\end{equation*}
Denote by $H_n^{(1)}(\cdot)$ the Hankel function of the first kind of order $n$. In polar coordinates ${\boldsymbol x}=(r\cos\theta, r\sin \theta)$,
$\phi$ and $\psi$ can be written as (see \cite{ColtonKress2013})
\begin{equation}\label{phiseries}
\phi(r,\theta)=\sum_{n=-\infty}^{\infty}a_n H_{|n|}^{(1)}(k_p r)e^{in\theta},\ \ \ r\geq R,
\end{equation}
\begin{equation}\label{psiseries}
\psi(r,\theta)=\sum_{n=-\infty}^{\infty}b_n H_{|n|}^{(1)}(k_s r)e^{in\theta},\ \ \ r\geq R.
\end{equation}
Since $\phi$ and $\psi$ satisfy boundary conditions $\frac{\partial {\boldsymbol \phi}}{\partial {\boldsymbol \nu}}+\frac{\partial\psi}{\partial {\boldsymbol \tau}}=-{\boldsymbol \nu}\cdot {\boldsymbol u}^{\textrm{inc}}$ and $\frac{\partial{\boldsymbol \phi}}{\partial {\boldsymbol \tau}}-\frac{\partial\psi}{\partial {\boldsymbol \nu}}=-{\boldsymbol \tau}\cdot {\boldsymbol u}^{\textrm{inc}}$ on $\partial B$, using the fact that $\frac{\partial }{\partial {\boldsymbol \nu}}=\frac{\partial }{\partial r}$ and $\frac{\partial}{\partial{\boldsymbol \tau}}=\frac{1}{R}\frac{\partial}{\partial\theta}$, we have
\begin{equation*}\label{series1}
\sum_{n=-\infty}^{\infty} k_p a_n H_{|n|}^{(1)'}(k_p R)e^{in\theta}+\frac{i}{R}\sum_{n=-\infty}^{\infty} nb_n H_{|n|}^{(1)}(k_s R)e^{i n \theta}=-{\boldsymbol \nu}(\theta)\cdot {\boldsymbol u}^{\textrm{inc}}(R,\theta),
\end{equation*}
\begin{equation*}\label{series2}
\frac{i}{R}\sum_{n=-\infty}^{\infty} n a_n H_{|n|}^{(1)}(k_p R)e^{in\theta}-\sum_{n=-\infty}^{\infty} k_sb_n H_{|n|}^{(1)'}(k_s R)e^{i n \theta}=-{\boldsymbol \tau}(\theta)\cdot {\boldsymbol u}^{\textrm{inc}}(R,\theta).
\end{equation*}
Multiplying these two equations by $e^{-in\theta}, n=-\infty,\cdots, \infty$, and integrating with respect to $\theta$, we obtain a linear system for the coefficients $a_n$ and $b_n$,
\begin{equation*}
 k_p H_{|n|}^{(1)'}(k_p R) a_n+\frac{in}{R}H_{|n|}^{(1)}(k_s R)b_n=\frac{1}{2\pi}\int_0^{2\pi} -\big({\boldsymbol \nu}(\theta)\cdot {\boldsymbol u}^{\textrm{inc}}(R,\theta)\big)e^{-in\theta}d\theta,
 \end{equation*}
 \begin{equation*}
 \frac{in}{R}H_{|n|}^{(1)}(k_p R) a_n-k_s H_{|n|}^{(1)'}(k_s R)b_n=\frac{1}{2\pi}\int_0^{2\pi} -\big({\boldsymbol \tau}(\theta)\cdot {\boldsymbol u}^{\textrm{inc}}(R,\theta)\big)e^{-in\theta}d\theta.
\end{equation*}
Solve $a_n$ and $b_n$ to obtain
\begin{eqnarray}\label{an}
 a_n&=&\frac{k_s H_{|n|}^{(1)'}(k_s R)}{2\pi C} \int_0^{2\pi} \big({\boldsymbol \nu}(\theta)\cdot {\boldsymbol u}^{\textrm{inc}}(R,\theta)\big)e^{-in\theta}d\theta\nonumber\\
 &&+ \frac{in H_{|n|}^{(1)}(k_s R)}{2\pi R C}\int_0^{2\pi} \big({\boldsymbol \tau}(\theta)\cdot {\boldsymbol u}^{\textrm{inc}}(R,\theta)\big)e^{-in\theta}d\theta,
 \end{eqnarray}
 \begin{eqnarray}\label{bn}
  b_n&=&\frac{in H_{|n|}^{(1)}(k_p R)}{2\pi RC} \int_0^{2\pi} \big({\boldsymbol \nu}(\theta)\cdot {\boldsymbol u}^{\textrm{inc}}(R,\theta)\big)e^{-in\theta}d\theta\nonumber\\
  &&- \frac{k_p H_{|n|}^{(1)'}(k_p R)}{2\pi C}\int_0^{2\pi} \big({\boldsymbol \tau}(\theta)\cdot {\boldsymbol u}^{\textrm{inc}}(R,\theta)\big)e^{-in\theta}d\theta,
\end{eqnarray}
where
\begin{equation*}\label{D}
C:=\frac{n^2}{R^2}H_{|n|}^{(1)}(k_p R)H_{|n|}^{(1)}(k_s R)-k_pk_s H_{|n|}^{(1)'}(k_p R)  H_{|n|}^{(1)'}(k_s R).
\end{equation*}
From \autoref{theorem1}, we know that $u_p^{\infty,{\boldsymbol 0}}(\theta)=ik_p \phi_\infty$ and $u_s^{\infty,{\boldsymbol 0}}(\theta)=ik_s \psi_\infty$. Based on the asymptotic behavior of the Hankel functions, \eqref{phiseries} and \eqref{psiseries}, the compressional part and shear part of the far field pattern ${\boldsymbol u}^B_\infty=(u_p^{\infty,{\boldsymbol 0}};u_s^{\infty,{\boldsymbol 0}})$ are
\begin{equation}\label{Farfield-Bp}
u_p^{\infty,{\boldsymbol 0}}(\theta)=ik_p \phi_\infty=ik_p e^{-i\frac{\pi}{4}}\sqrt{\frac{2}{\pi k_p}}\sum_{n=-\infty}^{\infty}a_n i^{-|n|}e^{in\theta},
\end{equation}
\begin{equation}\label{Farfield-Bs}
u_s^{\infty,{\boldsymbol 0}}(\theta)=ik_s \psi_\infty=ik_p e^{-i\frac{\pi}{4}}\sqrt{\frac{2}{\pi k_s}}\sum_{n=-\infty}^{\infty}b_n i^{-|n|}e^{in\theta},
\end{equation}
where $a_n$ and $b_n$ are given by \eqref{an} and \eqref{bn}, respectively.


The  plane incident wave ${\boldsymbol u}^{\textrm{inc}}$ can be written in the form
\begin{equation}\label{planeincident}
 {\boldsymbol u}^{\textrm{inc}}({x};d,a_p,a_s)=a_p {\boldsymbol d} e^{ik_p {\boldsymbol x}\cdot {\boldsymbol d}}+a_s {\boldsymbol d}^\perp e^{ik_s {\boldsymbol x}\cdot {\boldsymbol d}},\ \ \ a_p, a_s\in \mathbb{C}.
\end{equation}
Let $B_{\boldsymbol z} := \{{{\boldsymbol x}+{\boldsymbol z}}; {\boldsymbol x} \in B\}$ be a rigid disk centered at ${\boldsymbol z}\in\,\mathbb{R}^2$ with radius $R$. Then $B_{\boldsymbol z}$ is a translation of $B$ with respect to ${\boldsymbol z}$.
Denote the corresponding far field patterns for $B$ and $B_{\boldsymbol z}$ by
\begin{equation*}
{{\boldsymbol u}_\infty^B}(\hat{\boldsymbol x}; {\boldsymbol d},a_p,a_s):=(u_p^{\infty,{\boldsymbol 0}}(\hat{\boldsymbol x}); u_s^{\infty,{\boldsymbol 0}}(\hat{\boldsymbol x}))\ \ \textrm{and}\ \ {{\boldsymbol u}_{\infty}^{B_{\boldsymbol z}}}(\hat{\boldsymbol x}; {\boldsymbol d},a_p,a_s):=(u_p^{\infty,{\boldsymbol z}}(\hat{\boldsymbol x}); u_s^{\infty,{\boldsymbol z}}(\hat{\boldsymbol x})),
\end{equation*}
respectively. Then the  far field pattern for $B_{\boldsymbol z}$ can be computed using the following translation relations (see (2.13)-(2.16) in \cite{JiLiu}):
\begin{equation}\label{T1}
{\boldsymbol u}_\infty^{B_{\boldsymbol z}}(\hat{\boldsymbol x}; {\boldsymbol d},1,0)=T_1\ {{\boldsymbol u}_\infty^{B}}(\hat{\boldsymbol x}; {\boldsymbol d},1,0),
\end{equation}
with $T_1:={diag}(e^{i(k_p{\boldsymbol d}-k_p\hat{\boldsymbol x})\cdot {\boldsymbol z}} , e^{i(k_p{\boldsymbol d}-k_s\hat{\boldsymbol x})\cdot {\boldsymbol z}} )$, and
\begin{equation}\label{T2}
{\boldsymbol u}_\infty^{B_{\boldsymbol z}}(\hat{\boldsymbol x}; {\boldsymbol d},0,1)=T_2\ {{\boldsymbol u}_\infty^{B}}(\hat{\boldsymbol x}; {\boldsymbol d},0,1),
\end{equation}
with $T_2:={diag}(e^{i(k_s{\boldsymbol d}-k_p\hat{\boldsymbol x})\cdot {\boldsymbol z}} , e^{i(k_s{\boldsymbol d}-k_s\hat{\boldsymbol x})\cdot {\boldsymbol z}} )$.

%

\subsection{Extended sampling method}
Let $\mathbb{L}^2:=[L^2(\mathbb{S})]^2$. For any ${\boldsymbol g}:=(g_p;g_s), {\boldsymbol h}:=(h_p;h_s)\in \mathbb{L}^2$, we define an inner product for the Hilbert space $\mathbb{L}^2$ (see \cite{Arens2001}):
  \begin{equation}\label{innerproduct}
  \langle {\boldsymbol g}, {\boldsymbol h} \rangle:= \frac{\omega}{k_p} \int_\mathbb{S} g_p(\hat{\boldsymbol x})\overline{h_p(\hat{\boldsymbol x})} ds(\hat{\boldsymbol x})+\frac{\omega}{k_s}\int_\mathbb{S} g_s(\hat{\boldsymbol x})\overline{h_s(\hat{\boldsymbol x})} ds(\hat{\boldsymbol x}),\ \ \ {\boldsymbol g}, {\boldsymbol h}\in \mathbb{L}^2.
  \end{equation}
Given ${\boldsymbol g} \in \mathbb{L}^2$, the elastic Herglotz wave function with density ${\boldsymbol g}$ is defined as
\begin{equation}\label{herglotz}
v_{\boldsymbol g}({\boldsymbol x}):=\int_{\mathbb{S}} \bigg\{ \sqrt{\frac{k_p}{w}}{\boldsymbol d} e^{ik_p {\boldsymbol x}\cdot {\boldsymbol d}} g_p({\boldsymbol d})+ \sqrt{\frac{k_s}{w}}{\boldsymbol d}^\perp e^{ik_s {\boldsymbol x}\cdot d}g_s({\boldsymbol d})\bigg\} ds({\boldsymbol d}),\ \ \ {\boldsymbol x}\in \mathbb{R}^2.
\end{equation}

Denote by ${\boldsymbol U}({\boldsymbol x}), {\boldsymbol x}\in \mathbb{R}^2\setminus \overline{D}$ and ${\boldsymbol U}_\infty(\hat{\boldsymbol x}):=(U_p^\infty;U_s^\infty), \hat{\boldsymbol x}\in \mathbb{S}$ the scattered wave and far field pattern of the unknown scatterer $D$ due to an incident wave, respectively. The inverse problem {\bf IP-F} is to determine the location and size of the scatterer $D$ from the far field pattern ${\boldsymbol U}_\infty(\hat{\boldsymbol x})$.

Consider the far field equation
\begin{equation}\label{fareq}
(F_{\boldsymbol z}^*F_{\boldsymbol z})^{1/4}{\boldsymbol g}={\boldsymbol U}_\infty,
\end{equation}
where ${\boldsymbol g}\in \mathbb{L}^2$. The far field operator $F_{\boldsymbol z}: \mathbb{L}^2\rightarrow \mathbb{L}^2$ is defined as
\begin{eqnarray}\label{faroperator}
F_{\boldsymbol z}{\boldsymbol g}(\hat{\boldsymbol x})&:=&\int_\mathbb{S}{\boldsymbol u}_\infty^{B_{\boldsymbol z}}\big(\hat{\boldsymbol x},{\boldsymbol d},\sqrt{k_p/\omega}g_p({\boldsymbol d}),\sqrt{k_s/\omega}g_s({\boldsymbol d})\big)ds({\boldsymbol d})\nonumber\\
&=&\int_\mathbb{S}\bigg\{\sqrt{\frac{k_p}{\omega}}{\boldsymbol u}_\infty^{B_{\boldsymbol z}} (\hat{\boldsymbol x};{\boldsymbol d},1,0)g_p({\boldsymbol d})+\sqrt{\frac{k_s}{\omega}}{\boldsymbol u}_\infty^{B_{\boldsymbol z}}(\hat{\boldsymbol x};{\boldsymbol d},0,1)g_s({\boldsymbol d})\bigg\}ds({\boldsymbol d})\\
&=&\int_\mathbb{S}{\left( \begin{array}{cc}
\sqrt{\frac{k_p}{\omega}}u_p^{\infty,{\boldsymbol z}}(\hat{\boldsymbol x};{\boldsymbol d},1,0) & \sqrt{\frac{k_s}{\omega}}u_p^{\infty,{\boldsymbol z}}(\hat{\boldsymbol x};{\boldsymbol d},0,1)\nonumber\\
\sqrt{\frac{k_p}{\omega}}u_s^{\infty,{\boldsymbol z}}(\hat{\boldsymbol x};{\boldsymbol d},1,0) & \sqrt{\frac{k_s}{\omega}}u_s^{\infty,{\boldsymbol z}}(\hat{\boldsymbol x};{\boldsymbol d},0,1)
\end{array}
\right )} {\left( \begin{array}{c}
g_p({\boldsymbol d})\\
g_s({\boldsymbol d})
\end{array}
\right )}ds({\boldsymbol d}),
\end{eqnarray}
where ${\boldsymbol u}^{B_{\boldsymbol z}}_\infty(\hat{\boldsymbol x}; {\boldsymbol d},a_p,a_s):=(u_p^{\infty,{\boldsymbol z}};u_s^{\infty,{\boldsymbol z}})$ is the far field pattern of the rigid disk $B_{\boldsymbol z}$ due to the plane incident wave \eqref{planeincident}. As introduced in \autoref{appendixA}, ${\boldsymbol u}_\infty^{B_{\boldsymbol z}} (\hat{\boldsymbol x};{\boldsymbol d},1,0)$ and ${\boldsymbol u}_\infty^{B_{\boldsymbol z}} (\hat{\boldsymbol x};{\boldsymbol d},0,1)$ can be easily computed by using the series expansion and the translation properties \eqref{T1}-\eqref{T2}. Define the operator $G: [H^{1/2}(\partial B_z)]^2\rightarrow \mathbb{L}^2$ by
\begin{equation}\label{Gdef}
G{\boldsymbol h}:={\boldsymbol u}_\infty,
\end{equation}
where ${\boldsymbol u}_\infty$ is the far-field pattern of the solution to the elastic Dirichlet boundary value problem with boundary value ${\boldsymbol h}$. The uniqueness for the boundary value problem implies that $G$ is injective. The far field operator $F_{\boldsymbol z}$ has the following properties (see \cite{Arens2001}).\\

\begin{lemma}\label{lemma1}
Assume that $\omega^2$ is not a Dirichlet eigenvalue of $-\Delta^*$ in $B_{\boldsymbol z}$. The ranges of $G$ and $(F_{\boldsymbol z}^*F_{\boldsymbol z})^{1/4}$ coincide.
\end{lemma}

The following theorem is the main result for the far field equation \eqref{fareq}.\\

\begin{theorem}\label{imp_theorem}
Assume that $\omega^2$ is not a Dirichlet eigenvalue of $-\Delta^*$ in $B_{\boldsymbol z}$.
\begin{itemize}
\item[1.] If $D\subset B_{\boldsymbol z}$, the far field equation \eqref{fareq} is solvable.

\item[2.] If $D\cap B_{\boldsymbol z}=\emptyset$, the far field equation \eqref{fareq} has no solution.

\end{itemize}
\end{theorem}
\begin{proof}
1. Assume that $D\subset B_{\boldsymbol z}$. Define ${\boldsymbol h}:={\boldsymbol U}|_{\partial B_{\boldsymbol z}}$. According to the definition of $G$ in \eqref{Gdef}, ${\boldsymbol h}$ is a solution of $G{\boldsymbol h}={\boldsymbol U}_\infty$. Hence ${\boldsymbol U}_\infty$ is in the range of $G$. From Lemma \ref{lemma1}, ${\boldsymbol U}_\infty$ is also in the range of $(F_{\boldsymbol z}^*F_{\boldsymbol z})^{1/4}$.

2. Assume that $D\cap B_{\boldsymbol z}=\emptyset$ and the far field equation \eqref{fareq} has a solution. From Lemma \ref{lemma1}, ${\boldsymbol U}_\infty$ is in the range of $G$. Assume that $G{\boldsymbol h}={\boldsymbol U}_\infty$ for some ${\boldsymbol h}\in [H^{1/2}(\partial B_{\boldsymbol z})]^2$. Let ${\boldsymbol u}$ denote the radiating solution of the exterior Dirichlet problem for $B_{\boldsymbol z}$ with boundary value ${\boldsymbol h}$. Then ${\boldsymbol u}_\infty=G{\boldsymbol h}={\boldsymbol U}_\infty$. From Rellich's lemma, we can identify the radiating solution ${\boldsymbol v}:={\boldsymbol u}={\boldsymbol U}$ in $\mathbb{R}^2\setminus (\overline{B_{\boldsymbol z}}\cup \overline{D})$. Since ${\boldsymbol u}$ is defined in $\mathbb{R}^2\setminus \overline{B_{\boldsymbol z}}$, ${\boldsymbol U}$ is defined in $\mathbb{R}^2\setminus \overline{D}$ and ${\boldsymbol v}$ can be extended from $\mathbb{R}^2\setminus (\overline{B_{\boldsymbol z}}\cup \overline{D})$ into $\mathbb{R}^2$.
That is, ${\boldsymbol v}$ is an entire solution to the Navier equation.
Since ${\boldsymbol v}$ also satisfies the radiation condition, it must vanish identically in $\mathbb{R}^2$.
Then the total field coincides with the incident field and this leads to a contradiction
since the incident field ${\boldsymbol u}^{\textrm{inc}}$ cannot satisfy the boundary condition.
\end{proof}

In consistency with the linear sampling method and the use of partial far field data, instead of solving \eqref{fareq}, we consider the equation
\begin{equation}\label{fareq2}
\alpha {\boldsymbol g}_{\boldsymbol z}^\alpha+F_{\boldsymbol z}^*F_{\boldsymbol z}{\boldsymbol g}_{\boldsymbol z}^\alpha=F_{\boldsymbol z}^*{\boldsymbol U}_\infty,
\end{equation}
where $\alpha>0$ is the regularization parameter. \\

\begin{theorem}\label{imp_theorem2}
Let $F_{\boldsymbol z}$ be the far field operator \eqref{faroperator} and assume that $\omega^2$ is not a Dirichlet eigenvalue of $-\Delta^*$ in $B_{\boldsymbol z}$. Denote by ${\boldsymbol g}_{\boldsymbol z}$ the solution of $(F_{\boldsymbol z}^*F_{\boldsymbol z})^{1/4}{\boldsymbol g}={\boldsymbol U}_\infty$ if it exists. For $\alpha>0$, let ${\boldsymbol g}_{\boldsymbol z}^\alpha$ denote the solution of \eqref{fareq2} and ${\boldsymbol v}_{{\boldsymbol g}_{\boldsymbol z}^\alpha}$ denote the Herglotz wave function \eqref{herglotz} with kernel ${\boldsymbol g}_{\boldsymbol z}^\alpha$.
\begin{itemize}
\item[1.] If $D\subset B_{\boldsymbol z}$, then $\lim_{\alpha\rightarrow 0}{\boldsymbol v}_{{\boldsymbol g}_{\boldsymbol z}^\alpha}({\boldsymbol x})$ exists and
\begin{equation*}
c\|{\boldsymbol g}_{\boldsymbol z}\|^2\leq \lim_{\alpha\rightarrow 0} \|{\boldsymbol v}_{{\boldsymbol g}_{\boldsymbol z}^\alpha}({\boldsymbol x})\| \leq \|{\boldsymbol g}_{\boldsymbol z}\|^2
\end{equation*}
for some positive $c$ depending only on $B_{\boldsymbol z}$.

\item[2.] If $D\cap B_{\boldsymbol z}=\emptyset$, then $\lim_{\alpha\rightarrow 0}{\boldsymbol v}_{{\boldsymbol g}_{\boldsymbol z}^\alpha}({\boldsymbol x})\rightarrow \infty$.

\end{itemize}
\end{theorem}

\begin{proof}
With the property of $F_{\boldsymbol z}$ in Lemma \ref{lemma1}, and also \autoref{imp_theorem}, the proof is analogous to the proof of Theorem 5.41 and Corollary 5.42 of \cite{ColtonKress2013}.
\end{proof}

Now we are ready to present the extended sampling method (ESM) for {\bf IP-F}. Let $\Omega$ be a domain containing $D$. For a sampling point ${\boldsymbol z} \in \Omega$, 
by \autoref{imp_theorem2}, one expects that the solution $\|{\boldsymbol g}^\alpha_{\boldsymbol z}\|_{L^2(\mathbb{S})}$ of \eqref{fareq2} is relatively large when $D$ is outside $B_{\boldsymbol z}$
and relatively small when $D$ is inside $B_{\boldsymbol z}$.
Consequently, an approximation of the location and size of the scatterer $D$ can be obtained by plotting
$\|{\boldsymbol g}^\alpha_{\boldsymbol z}\|_{L^2(\mathbb{S})}$ for all sampling points ${\boldsymbol z} \in \Omega$.\\

\begin{itemize}
\item[ ] \hspace{-1cm}{\bf{The Extended Sampling Method for IP-F}}{\em
\item[1.] Generate a set $T$ of sampling points for $\Omega$ which contains $D$.
\item[3.] For each sampling point ${\boldsymbol z} \in T$,
	\begin{itemize}
	\item[a.] compute ${\boldsymbol u}^{B_{\boldsymbol z}}_\infty(\hat{\boldsymbol x}; {\boldsymbol d},1,0)$ and ${\boldsymbol u}^{B_{\boldsymbol z}}_\infty(\hat{\boldsymbol x}; {\boldsymbol d},0,1)$, $\hat{\boldsymbol x}, {\boldsymbol d}\in \mathbb{S}$;
	\item[b.] choose a Tikhonov regularization parameter $\alpha$ and compute an approximate solution ${\boldsymbol g}^\alpha_{\boldsymbol z}$ to \eqref{fareq2}.
	\end{itemize}
\item[4.] Find the global minimum point ${\boldsymbol z}^*\in T$ for $\|{\boldsymbol g}^\alpha_{\boldsymbol z}\|_{L^2(\mathbb{S})}$.
\item[5.] Choose $B_{{\boldsymbol z}^*}$ to be the reconstruction for $D$.}
\end{itemize}

\begin{remark}
As the Multilevel ESM for {\bf IP-P}, one can use a similar multilevel technique to choose a suitable radius $R$ of the sampling disks.
\end{remark}

\section{Numerical examples}\label{section4}
We present some numerical examples to show the performance of the proposed method. The synthetic far field data is generated using the boundary integral equation method in \cite{WangQiu2017}.

We consider three obstacles: a rigid pear given by
\begin{equation*}
(1+0.15\cos 3\theta)\big(\cos \theta, \sin \theta\big)+(-2,3), \ \ \ \theta\in [0,2\pi),
\end{equation*}
a cavity peanut given by
\begin{equation*}
1.5\sqrt{\cos^2 \theta+0.25\sin^2 \theta}\big(\cos \theta, \sin \theta\big)+(-2,3), \ \ \ \theta\in [0,2\pi),
\end{equation*}
and a kite with impedance boundary condition ($\sigma=2$) given by
\begin{equation*}
\big(1.5\sin \theta, \cos \theta +0.65 \cos 2\theta-0.65\big)+\big(-2,3\big),\ \ \ \theta\in [0,2\pi).
\end{equation*}

For all numerical examples, $\omega=\pi, \mu=1, \lambda=2$. The incident plane wave is
${\boldsymbol u}^{inc}({\boldsymbol x})={\boldsymbol d} e^{ik_p{\boldsymbol x}\cdot {\boldsymbol d}}+{\boldsymbol d}^\perp e^{ik_s {\boldsymbol x}\cdot {\boldsymbol d}}$ with ${\boldsymbol d}=(1/2;\sqrt{3}/2)$.

For the inverse problem {\bf IP-P}, which uses only the compressional part or shear part of the far field pattern, the synthetic data is a $52\times 1$ vector ${\boldsymbol f}_P=U_p^\infty(\hat{\boldsymbol x}_j, {\boldsymbol d})$ or ${\boldsymbol f}_P=U_s^\infty(\hat{\boldsymbol x}_j,{\boldsymbol d})$ with $52$ observation directions $\hat{\boldsymbol x}_j, j=1,2,\cdots, 52$, uniformly distributed on the unit circle.

For the inverse problem {\bf IP-F}, ${\boldsymbol f}_F=\big(U_p^\infty(\hat{\boldsymbol x}_j,{\boldsymbol d}); U_s^\infty(\hat{\boldsymbol x}_j,{\boldsymbol d})\big), j=1,2,\cdots, 52$, is a $104\times 1$ vector.

\subsection{Examples for IP-P}
Let $\Omega=[-5,5]\times [-5, 5]$ and choose the samplings points to be
\begin{equation}\label{mesh}
T:=\{(-5+0.1m, -5+0.1n), \quad m,n=0, 1, \cdots, 100\}.
\end{equation}

For each mesh point ${\boldsymbol z} \in T$, we use the Tikhonov regularization with a fixed parameter $\alpha=10^{-5}$.
Equation \eqref{fe} leads to a linear system $A^{\boldsymbol z}\vec{\boldsymbol g}_{\boldsymbol z}={\boldsymbol f}_P$, where $A^{\boldsymbol z}$ is the matrix given by
\[
A^{\boldsymbol z}_{l,j}=e^{i k{\boldsymbol z}\cdot (\hat{\boldsymbol x}_j-\hat{\boldsymbol x}_l)}{\boldsymbol u}^B_{\infty}(\hat{\boldsymbol x}_l,\hat{\boldsymbol x}_j), \quad l, j=1,2,\cdots,52.
\]
The regularized solution is given by
\begin{equation*}
\vec{\boldsymbol g}^\alpha_{\boldsymbol z}\approx \big((A^{\boldsymbol z})^\ast A^{\boldsymbol z}+\alpha I\big)^{-1}(A^{\boldsymbol z})^\ast {\boldsymbol f}_P,
\end{equation*}
where $I$ is the identity matrix.
We plot the contours for the indicator function
\begin{equation}\label{Iz}
 I_{\boldsymbol z} = \frac{\|\vec{\boldsymbol g}^\alpha_{\boldsymbol z}\|_{l^2}}{\max\limits_{{\boldsymbol z}\in T}\|\vec{\boldsymbol g}^\alpha_{\boldsymbol z}\|_{l^2}}
\end{equation}
for all the sampling points ${\boldsymbol z}\in T$.

\autoref{fig1} shows the results for the rigid pear. The asterisks are the minimum locations of $I_{\boldsymbol z}$.
The solid lines are the reconstructions and the red dashed lines are the exact boundaries.
Similar results for the cavity peanut and the kite with impedance boundary condition are shown in \autoref{fig2} and \autoref{fig3}, respectively.

\begin{figure}[h!]
\begin{minipage}[t]{\linewidth}
\begin{minipage}[t]{0.40\linewidth}
\includegraphics[angle=0, width=\textwidth]{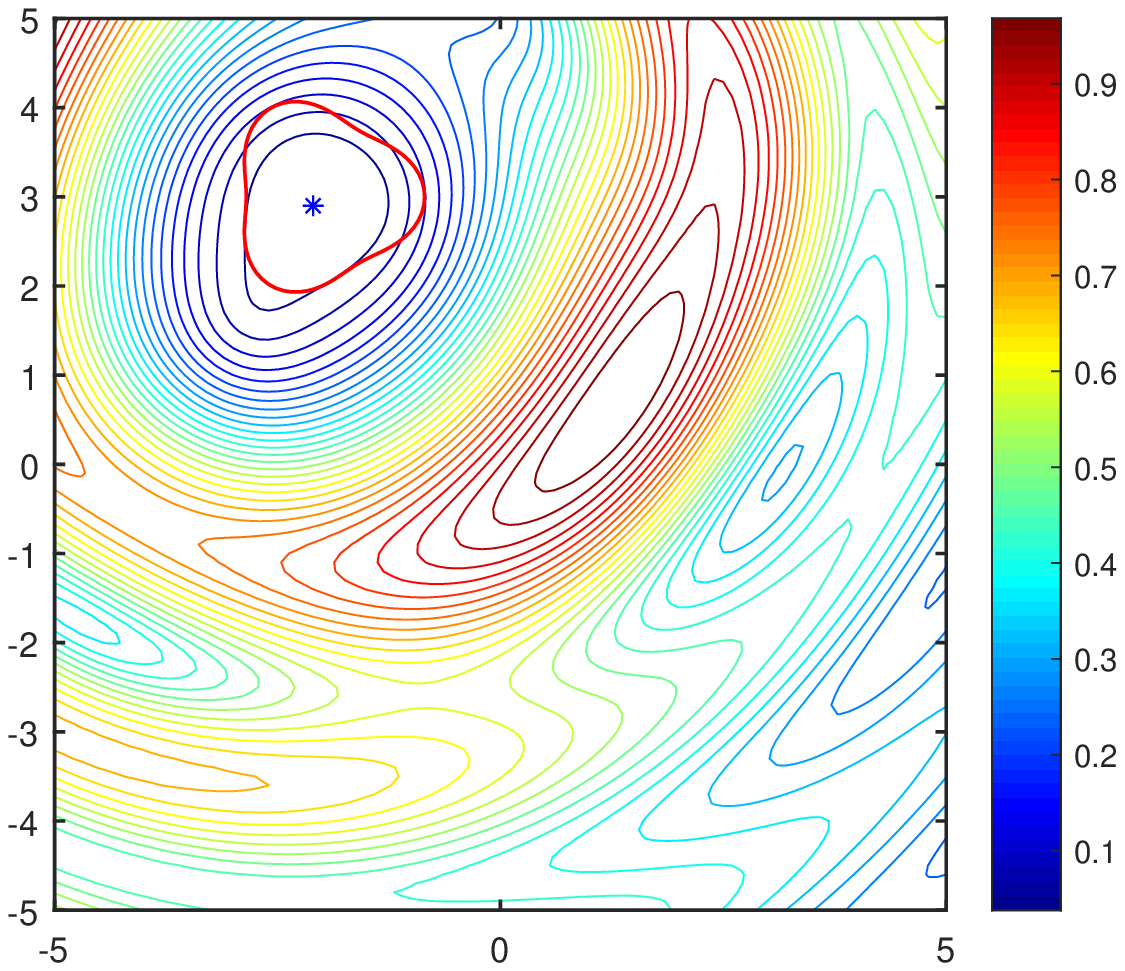}
\vspace{0.5cm}
\end{minipage}
\hspace{1.5cm}
\begin{minipage}[t]{0.40\linewidth}
\includegraphics[angle=0, width=0.83\textwidth]{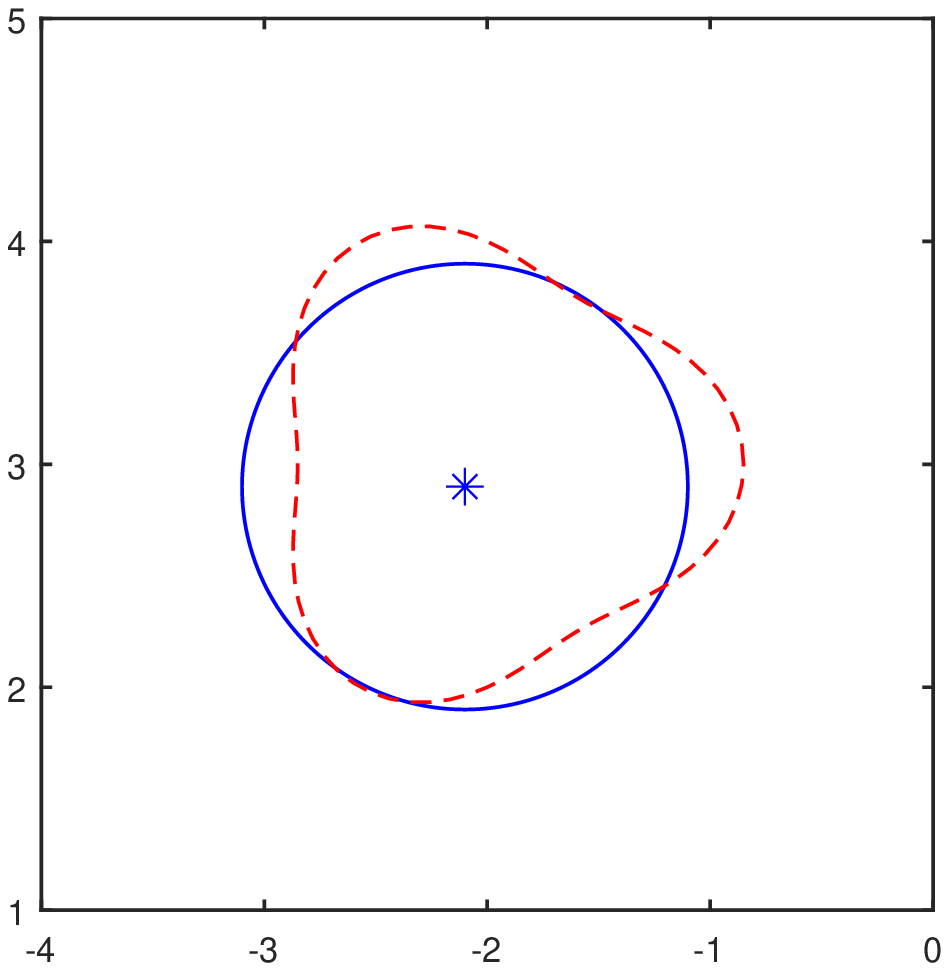}
\end{minipage}
\end{minipage}

\begin{minipage}[t]{\linewidth}
\begin{minipage}[t]{0.40\linewidth}
\includegraphics[angle=0, width=\textwidth]{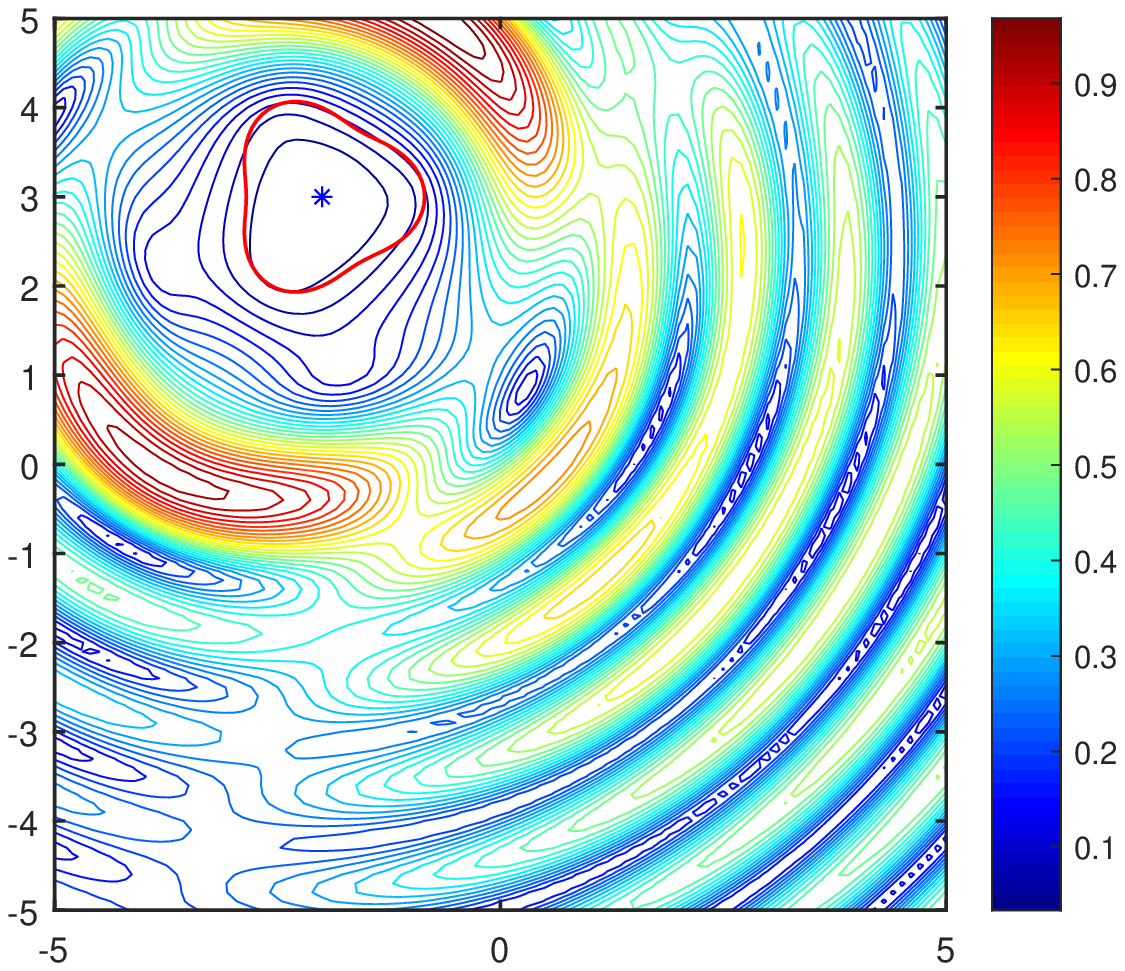}
\end{minipage}
\hspace{1.5cm}
\begin{minipage}[t]{0.40\linewidth}
\includegraphics[angle=0, width=0.83\textwidth]{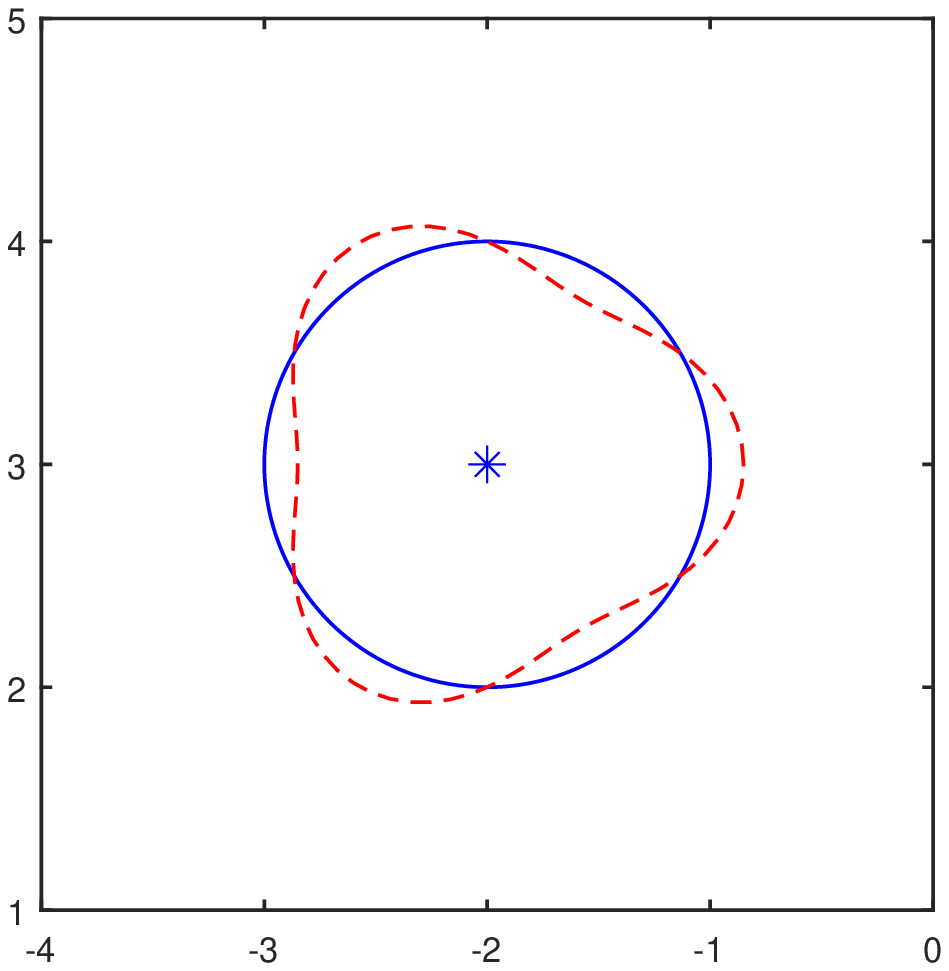}
\end{minipage}
\end{minipage}
\caption{\label{fig1} Reconstructions of the rigid pear.  Top left: contour plot of $I_z$ using compressional part of the far field pattern; Top right: reconstruction using the compressional part of the far field pattern; Bottom left: contour plot of $I_z$ using the shear part of the far field pattern; Bottom right: reconstruction using the shear part of the far field pattern.}
\end{figure}

\begin{figure}[ht!]
\begin{minipage}[t]{\linewidth}
\begin{minipage}[t]{0.40\linewidth}
\includegraphics[angle=0, width=\textwidth]{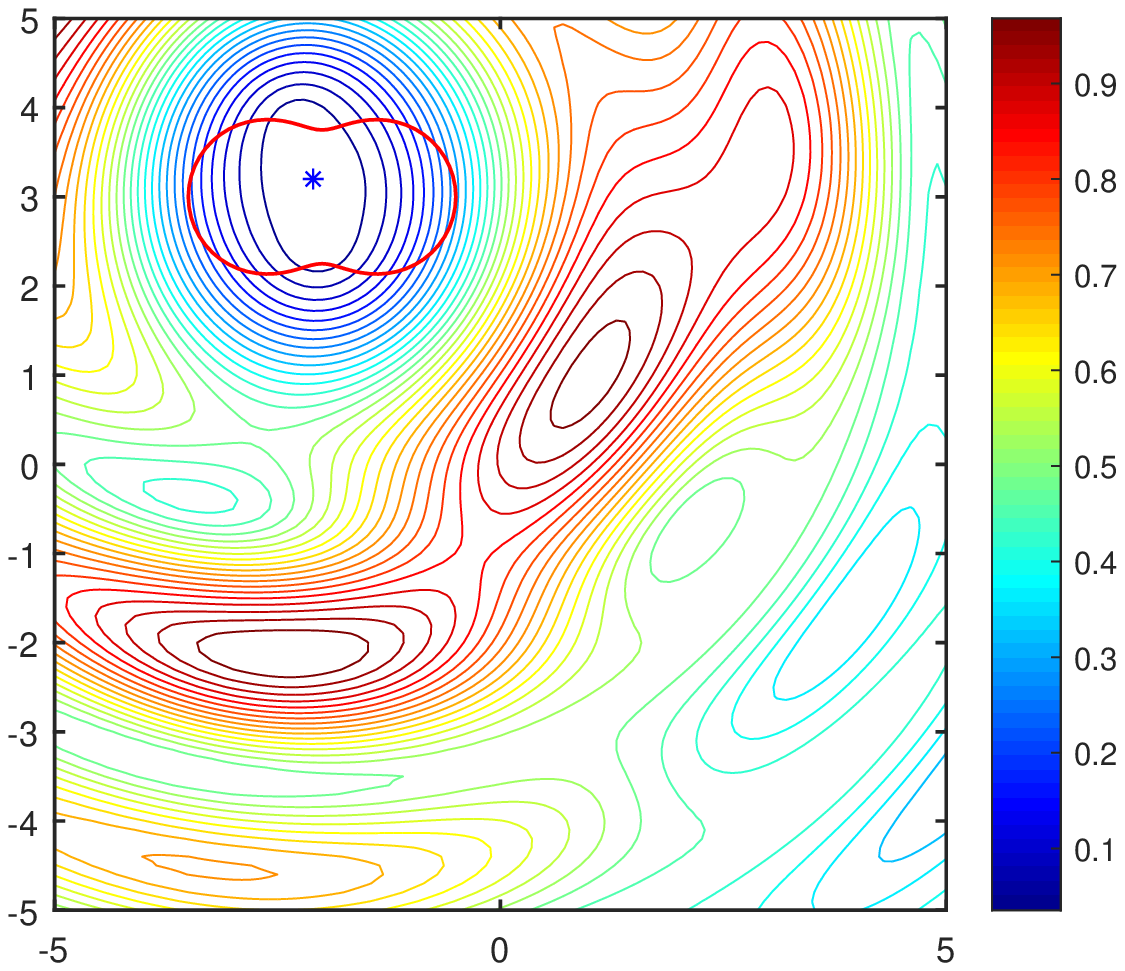}
\vspace{0.5cm}
\end{minipage}
\hspace{1.5cm}
\begin{minipage}[t]{0.40\linewidth}
\includegraphics[angle=0, width=0.83\textwidth]{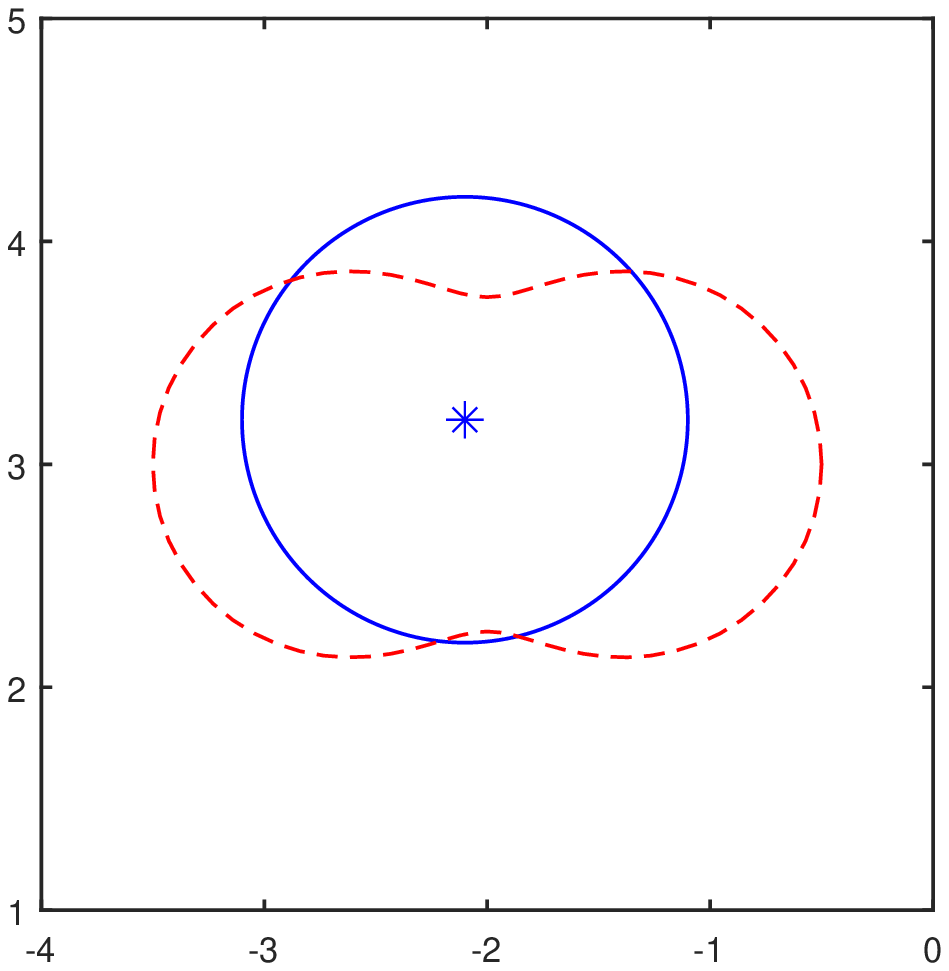}
\end{minipage}
\end{minipage}

\begin{minipage}[t]{\linewidth}
\begin{minipage}[t]{0.40\linewidth}
\includegraphics[angle=0, width=\textwidth]{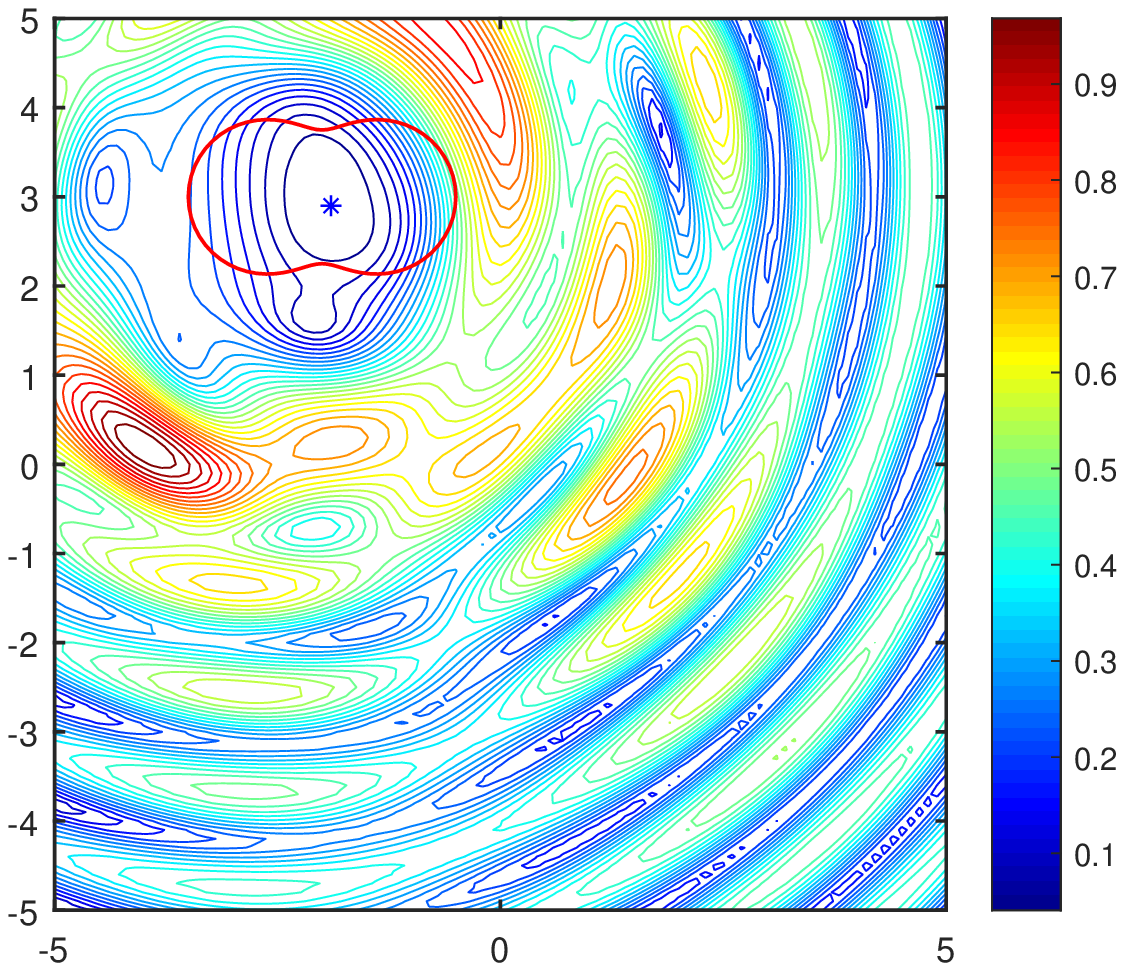}
\end{minipage}
\hspace{1.5cm}
\begin{minipage}[t]{0.40\linewidth}
\includegraphics[angle=0, width=0.83\textwidth]{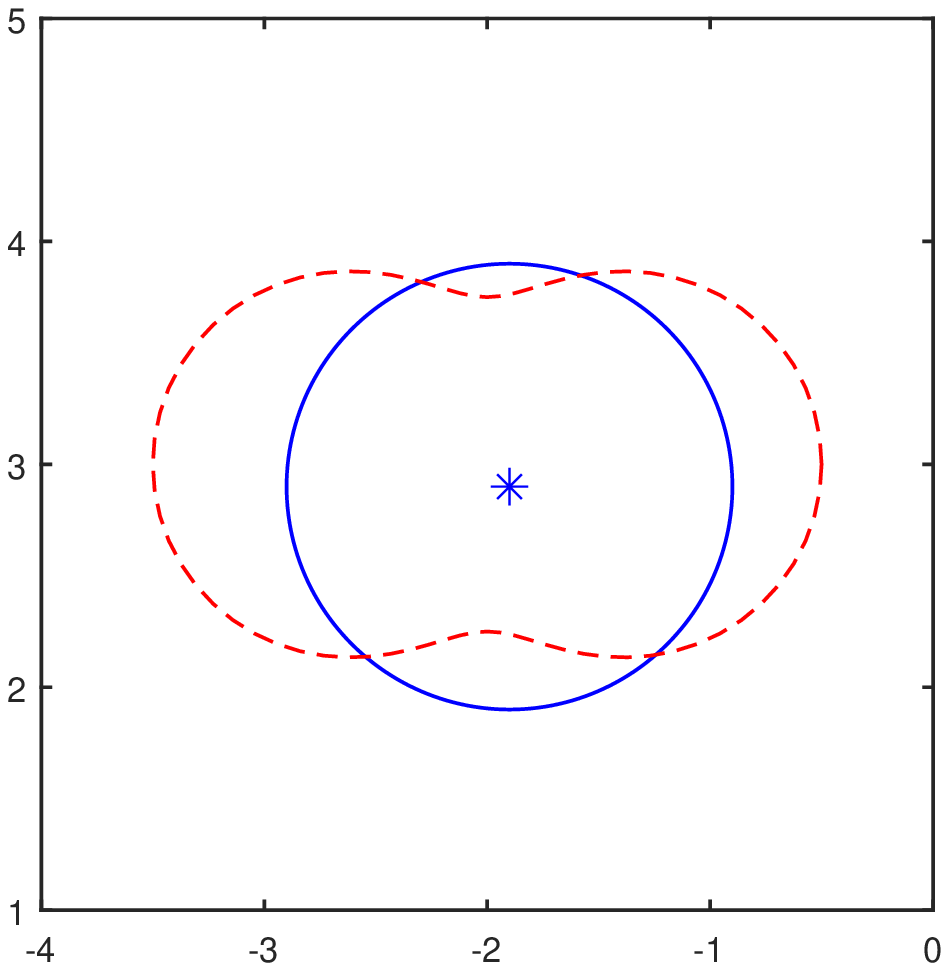}
\end{minipage}
\end{minipage}
\caption{\label{fig2} Reconstructions of the cavity peanut. Top left: contour plot of $I_{\boldsymbol z}$ using the compressional part of the far field pattern; Top right: reconstruction using the compressional part of the far field pattern; Bottom left: contour plot of $I_{\boldsymbol z}$ using the shear part of the far field pattern; Bottom right: reconstruction using the shear part of the far field pattern.}
\end{figure}


\begin{figure}[ht!]
\begin{minipage}[t]{\linewidth}
\begin{minipage}[t]{0.40\linewidth}
\includegraphics[angle=0, width=\textwidth]{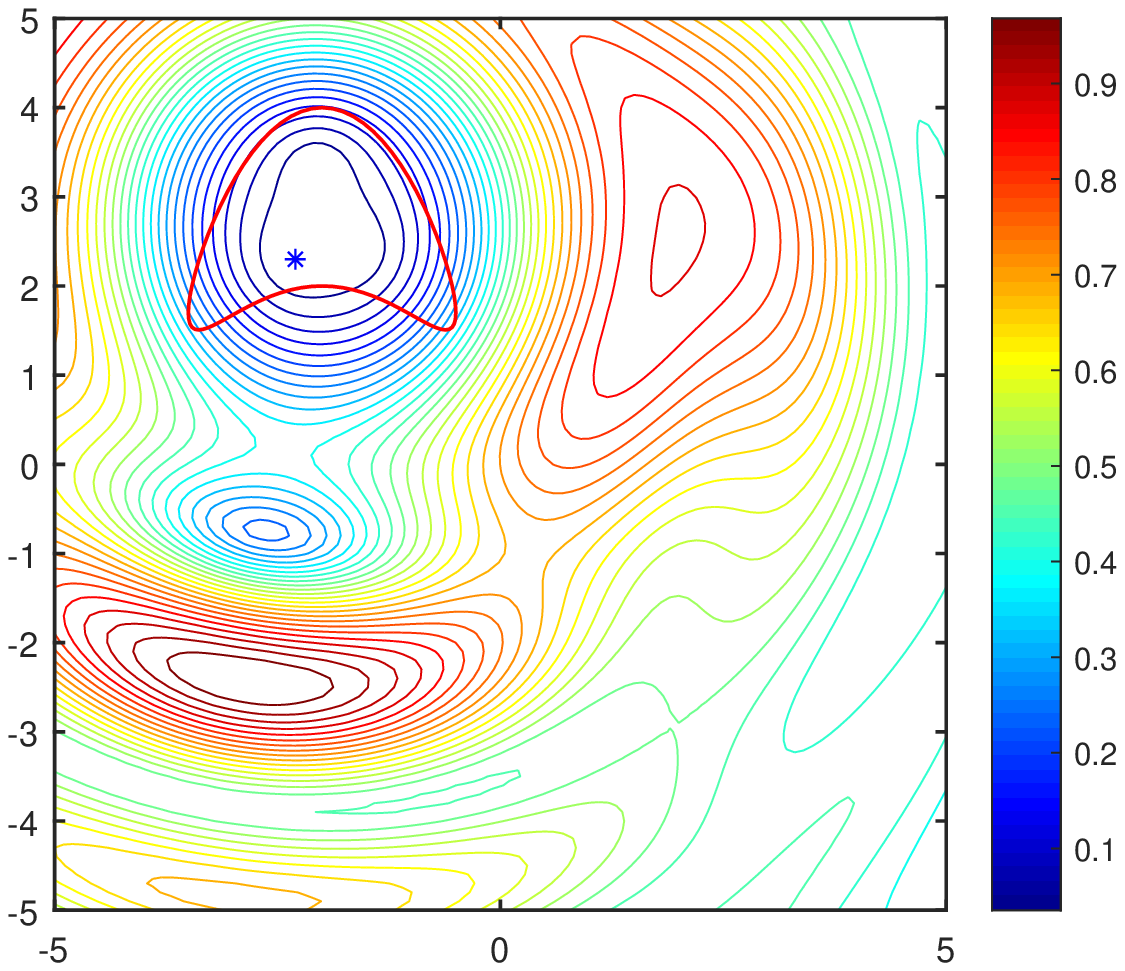}
\vspace{0.5cm}
\end{minipage}
\hspace{1.5cm}
\begin{minipage}[t]{0.40\linewidth}
\includegraphics[angle=0, width=0.83\textwidth]{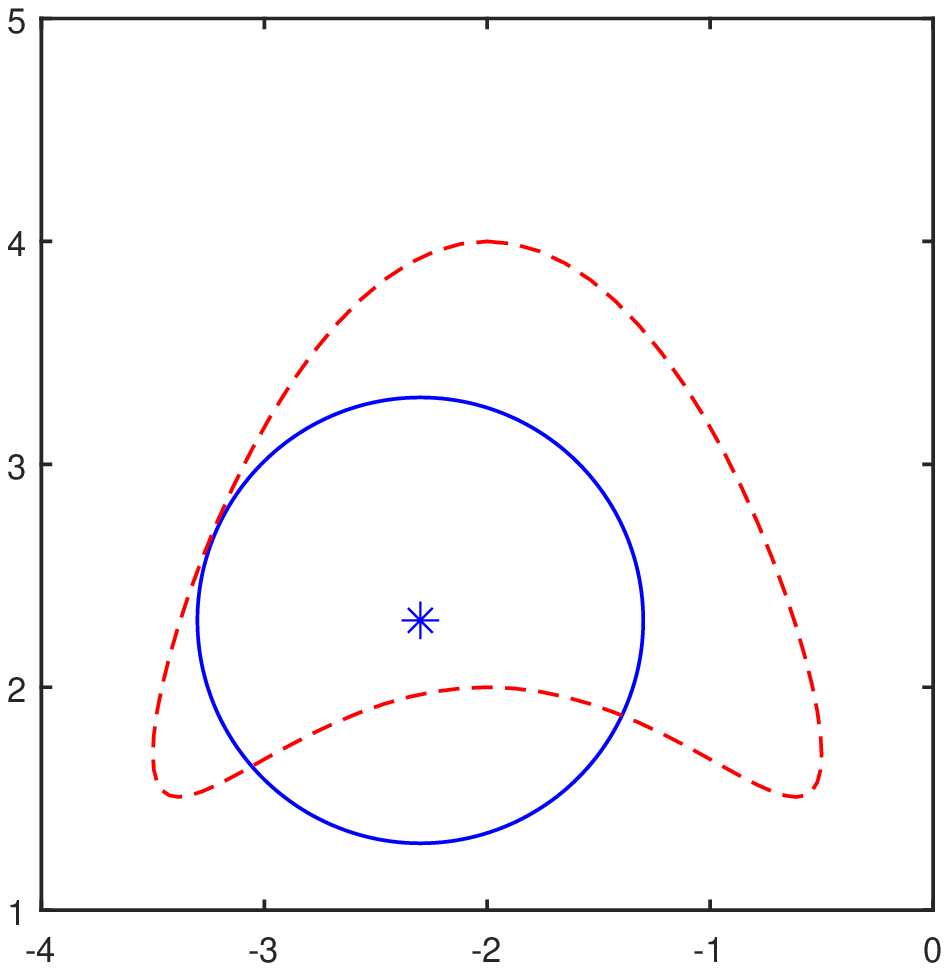}
\end{minipage}
\end{minipage}

\begin{minipage}[t]{\linewidth}
\begin{minipage}[t]{0.40\linewidth}
\includegraphics[angle=0, width=\textwidth]{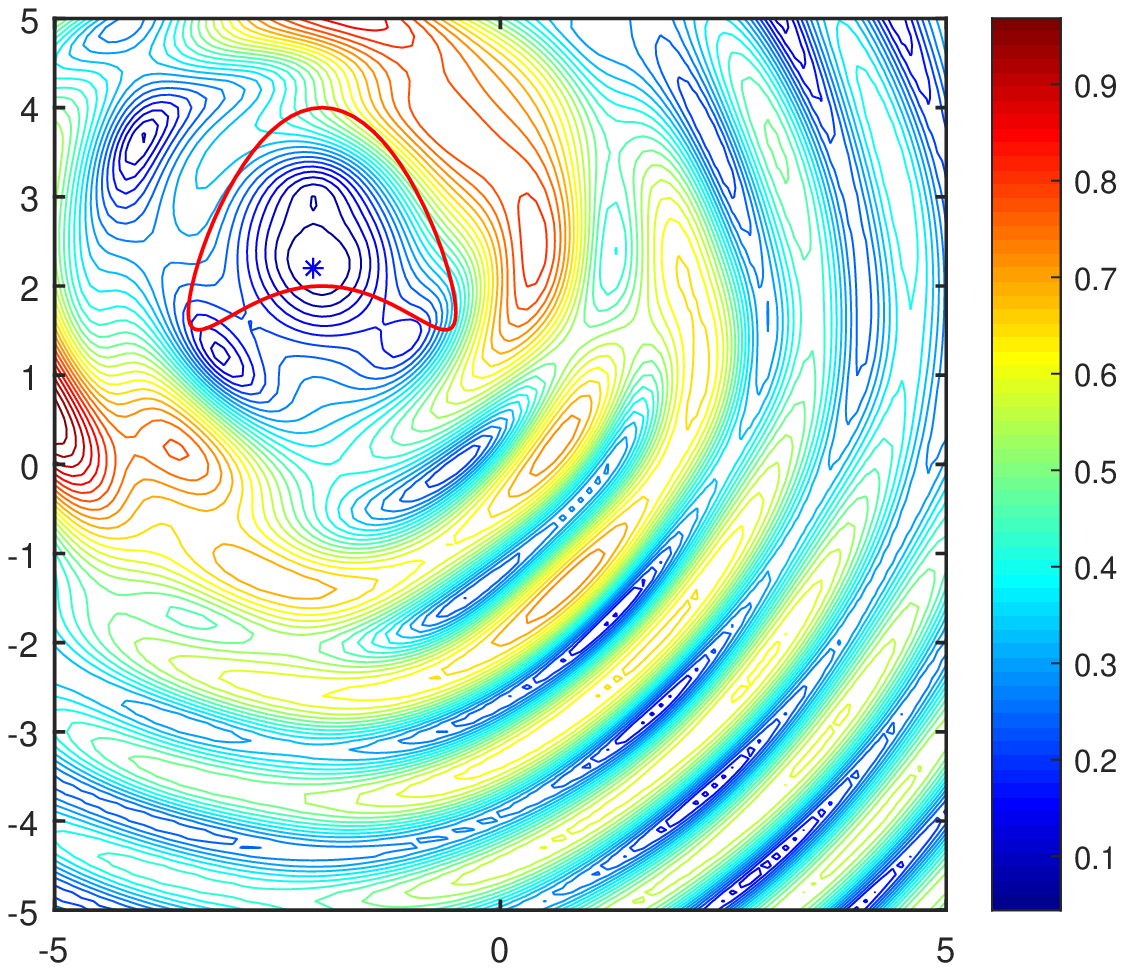}
\end{minipage}
\hspace{1.5cm}
\begin{minipage}[t]{0.40\linewidth}
\includegraphics[angle=0, width=0.83\textwidth]{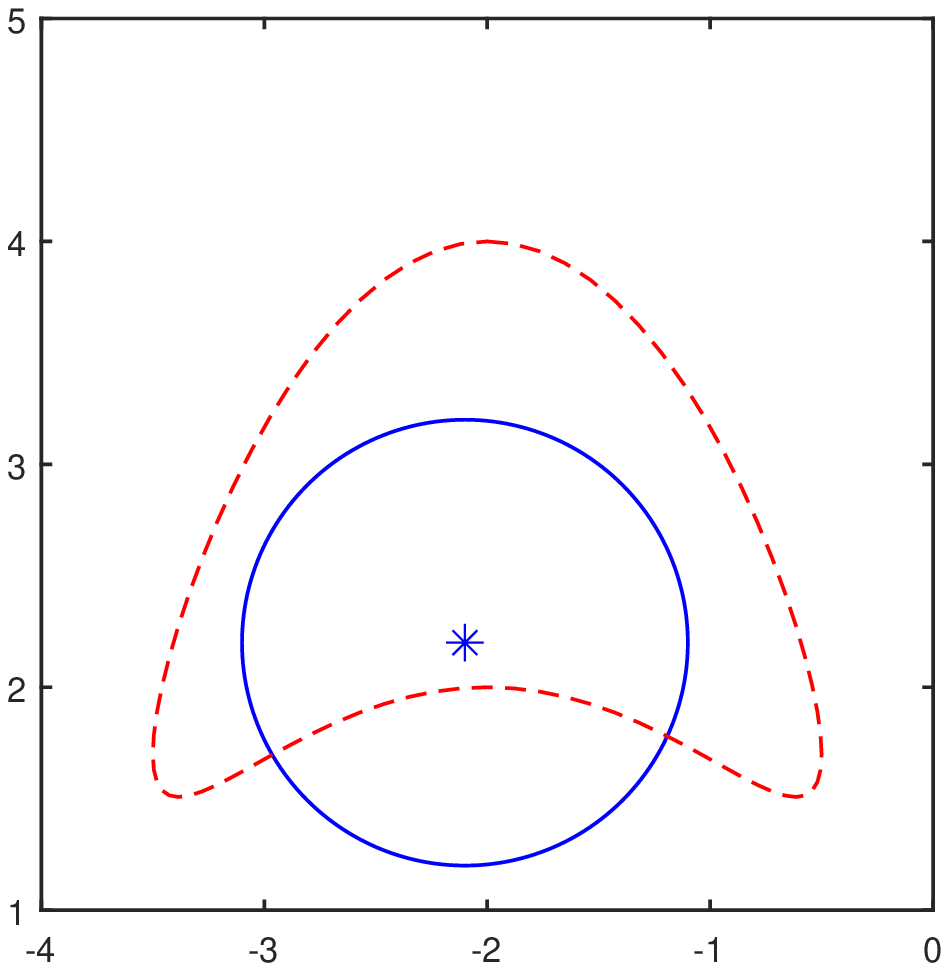}
\end{minipage}
\end{minipage}
\caption{\label{fig3} Reconstructions of the impedance kite. Top left: contour plot of $I_z$ using the compressional part of the far field pattern; Top right: reconstruction the the compressional part of the far field pattern; Bottom left: contour plot of $I_z$ using the shear part of the far field pattern; Bottom right: construction using the shear part of the far field pattern.}
\end{figure}

Since the size of the scatterer is not known in advance, one can determine the radius $R$ of the sampling disks using the multilevel ESM.
We start with a large sampling disk ($R=2.4$) and decrease the radius until a suitable $R$ is found.

For the rigid pear, $R$ is found to be $0.6$ using either the compressional part or shear part of the far field pattern.
For the cavity peanut and impedance kite, the radius $R$ is $0.3$ using the compressional part of the far field pattern.
The radius $R$ is $1.2$ using the shear part of the far field pattern.

\autoref{fig4} shows the reconstructions of the multilevel ESM for the pear with Dirichlet boundary condition, the peanut with Neumann boundary condition, and the kite with the impedance boundary condition  ($\sigma=2$).

\begin{figure}[h!]
\begin{minipage}[t]{\linewidth}
\begin{minipage}[t]{0.30\linewidth}
\includegraphics[angle=0, width=\textwidth]{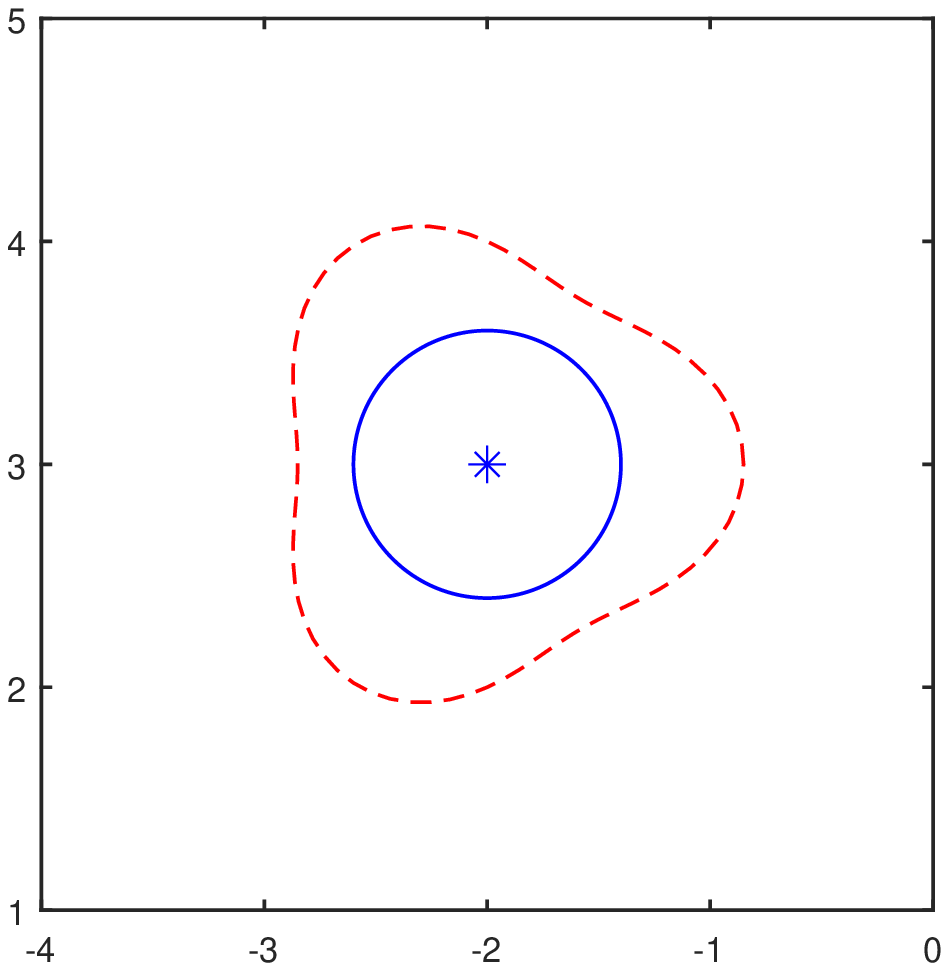}
\vspace{0.3cm}
\end{minipage}
\hspace{0.2cm}
\begin{minipage}[t]{0.30\linewidth}
\includegraphics[angle=0, width=\textwidth]{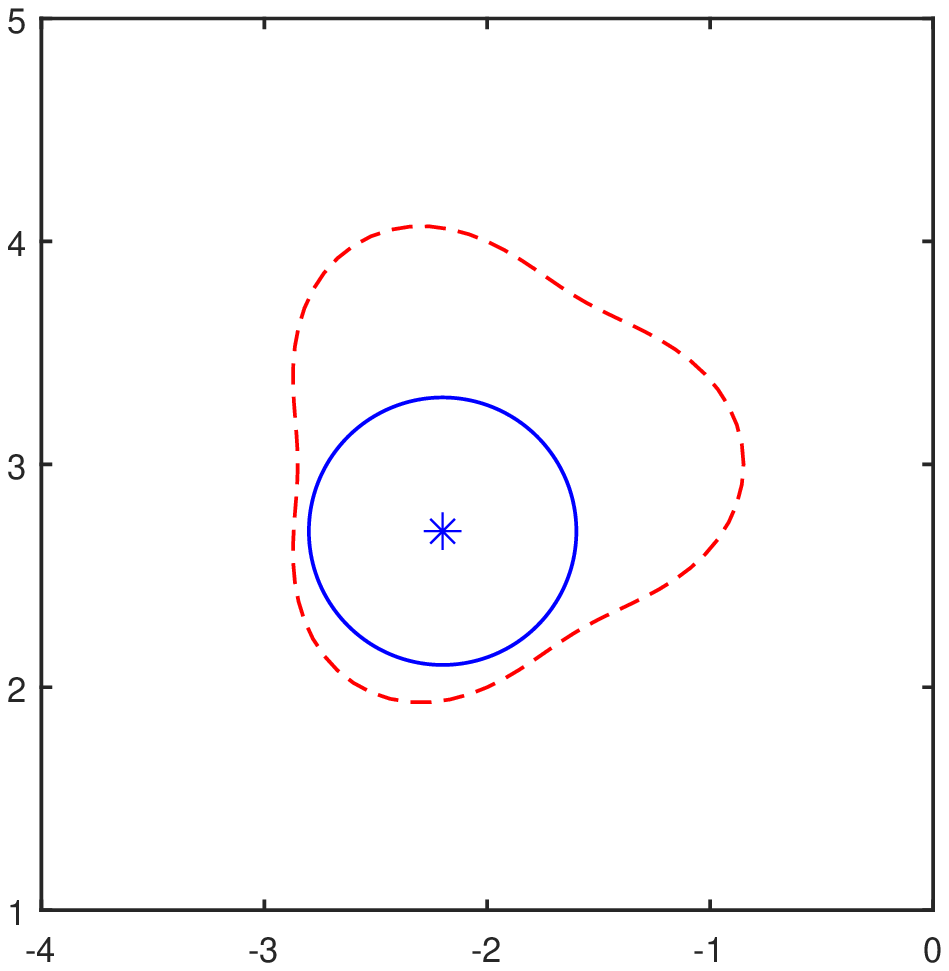}

\end{minipage}
\hspace{0.2cm}
\begin{minipage}[t]{0.30\linewidth}
\includegraphics[angle=0, width=\textwidth]{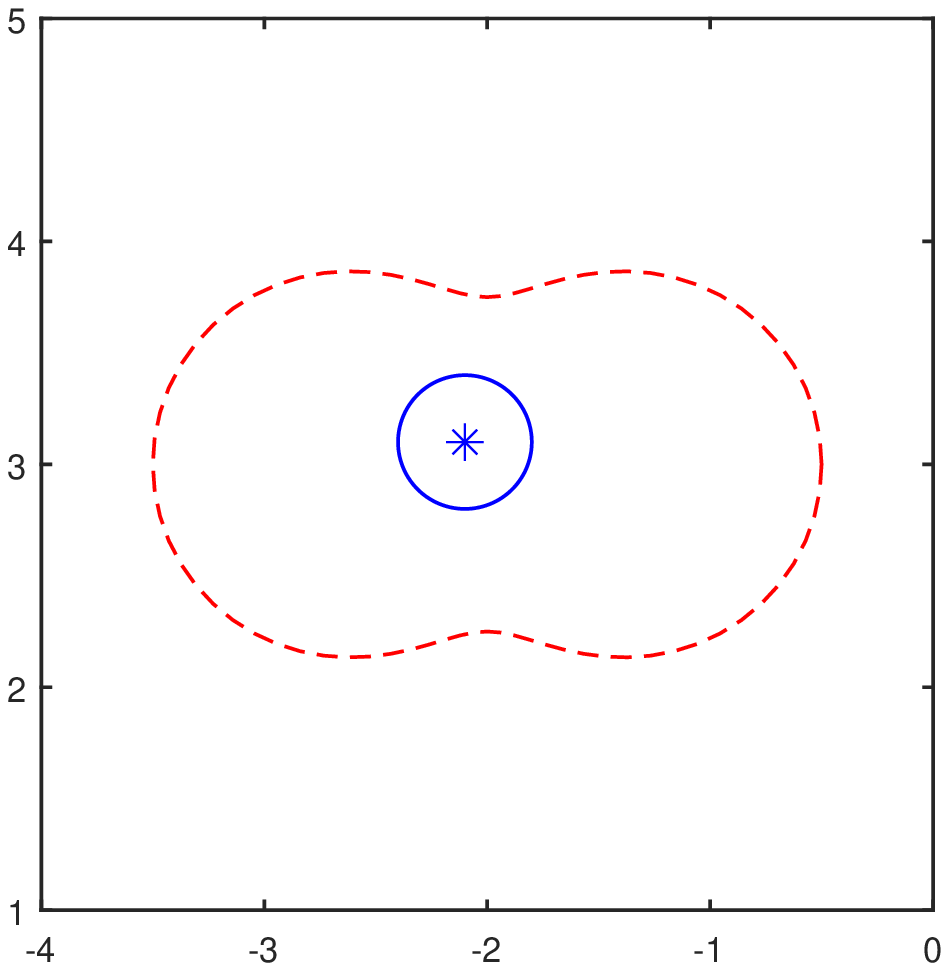}

\end{minipage}
\end{minipage}

\begin{minipage}[t]{\linewidth}
\begin{minipage}[t]{0.30\linewidth}
\includegraphics[angle=0, width=\textwidth]{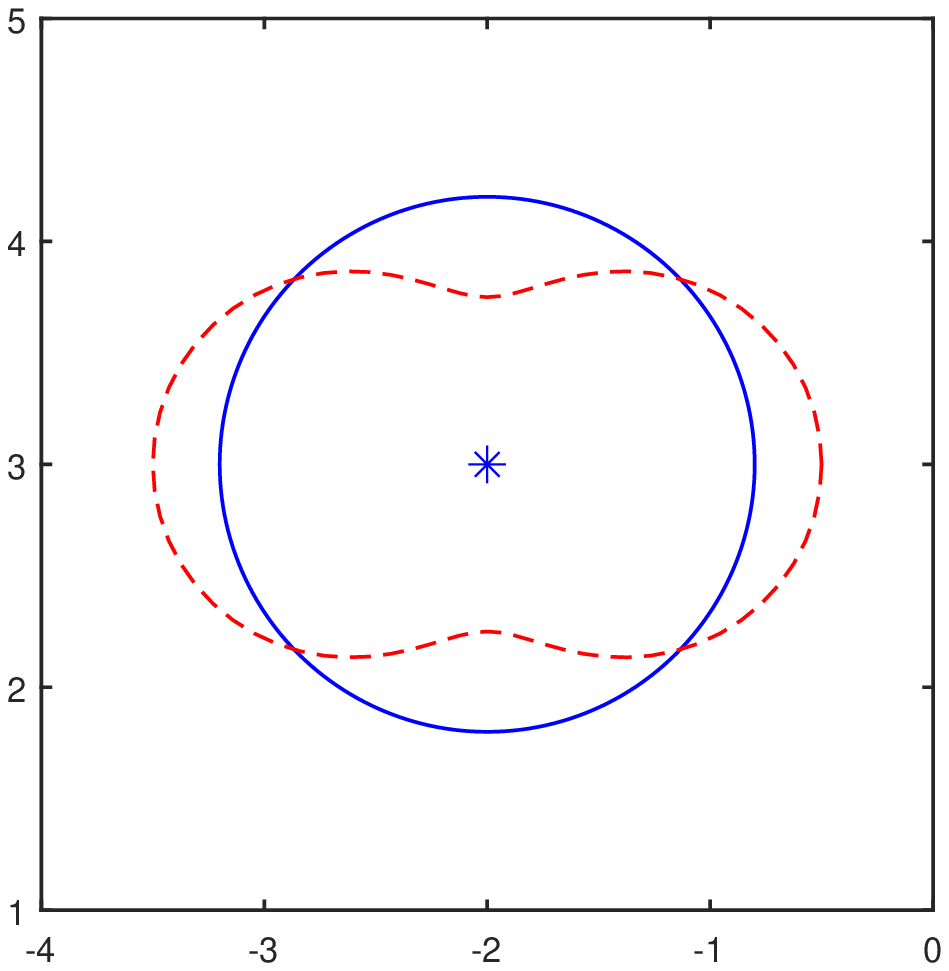}
\vspace{-0.9cm}

\end{minipage}
\hspace{0.2cm}
\begin{minipage}[t]{0.30\linewidth}
\includegraphics[angle=0, width=\textwidth]{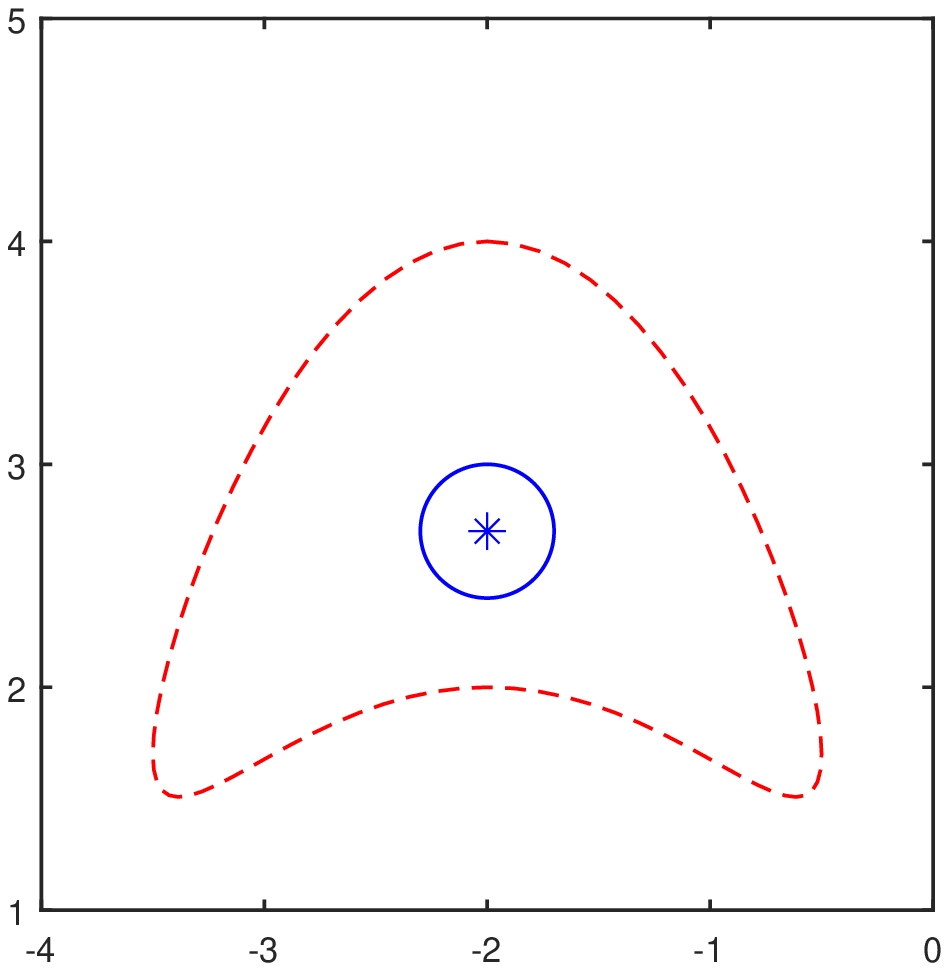}
\vspace{-0.9cm}

\end{minipage}
\hspace{0.2cm}
\begin{minipage}[t]{0.30\linewidth}
\includegraphics[angle=0, width=\textwidth]{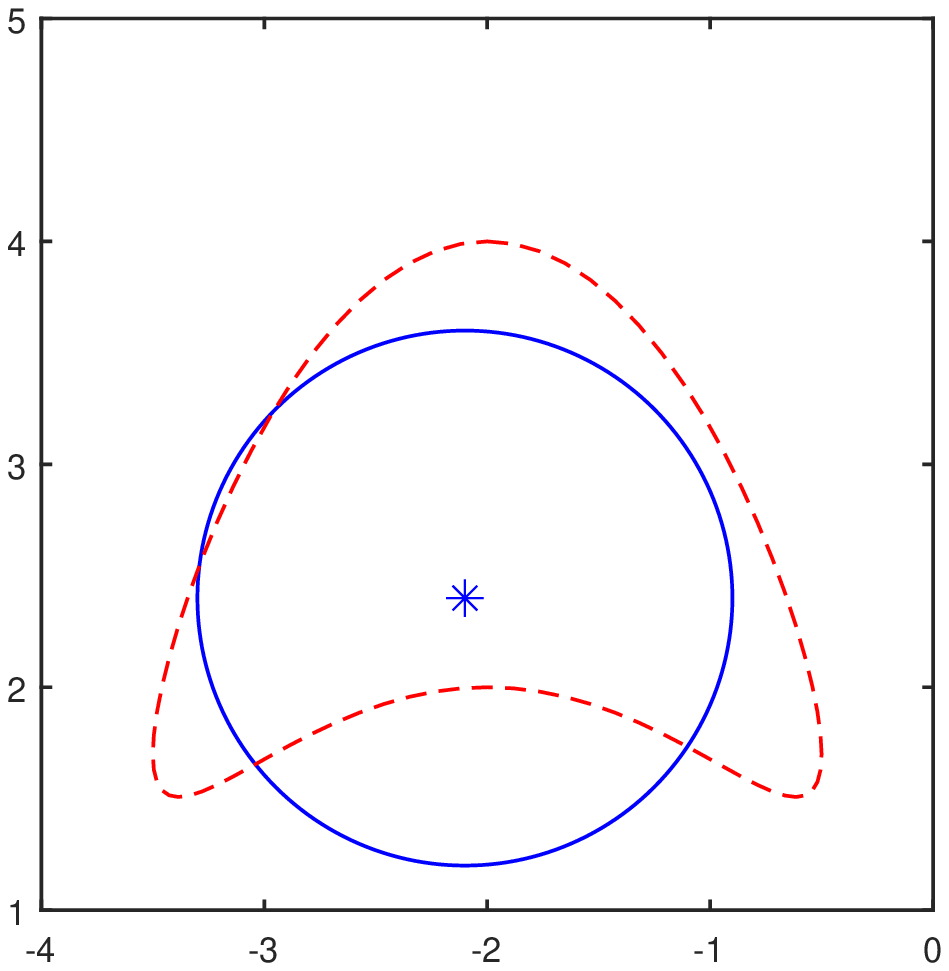}
\vspace{-0.9cm}

\end{minipage}
\end{minipage}
\caption{\label{fig4} The reconstruction using the multilevel ESM. Top left: the rigid pear using the compressional part of the far field pattern; Top middle: the rigid pear using the shear part; Top right: the cavity peanut using the compressional part; Bottom left: the cavity peanut using the shear part; Bottom middle: the impedance kite using the compressional part; Bottom right: the impedance kite using the shear part.}
\end{figure}

\subsection{Examples for IP-F}
We use the same $T$ defined in \eqref{mesh} and the Tikhonov regularization with $\alpha=10^{-5}$.
Equation \eqref{fareq2} leads to the following linear system
\begin{equation*}
\big(B_{\boldsymbol z}^*B_{\boldsymbol z}+\alpha I\big)\vec{\boldsymbol g}^\alpha_{\boldsymbol z}=B_{\boldsymbol z}^*{\boldsymbol f}_F,
\end{equation*}
where
\begin{equation*}
B_{\boldsymbol z}:={\left( \begin{array}{cc}
\bigg(\sqrt{\frac{k_p}{w}}u_p^{\infty,{\boldsymbol z}}(\hat{\boldsymbol x}_l;\hat{\boldsymbol x}_j,1,0)\bigg)_{52\times52} & \bigg(\sqrt{\frac{k_s}{w}}u_p^{\infty,{\boldsymbol z}}(\hat{\boldsymbol x}_l;\hat{\boldsymbol x}_j,0,1)\bigg)_{52\times52}\nonumber\\
\bigg(\sqrt{\frac{k_p}{w}}u_s^{\infty,{\boldsymbol z}}(\hat{\boldsymbol x}_l;\hat{\boldsymbol x}_j,1,0)\bigg)_{52\times52} & \bigg(\sqrt{\frac{k_s}{w}}u_s^{\infty,z}(\hat{\boldsymbol x}_l;\hat{\boldsymbol x}_j,0,1)\bigg)_{52\times52}
\end{array}
\right )}.
\end{equation*}
According to the definition of the inner product \eqref{innerproduct}, we have
\begin{equation*}
B_{\boldsymbol z}^*=D_{ps}^{-1}\overline{B_{\boldsymbol z}}^T D_{ps}
\end{equation*}
with the diagonal matrix $D_{ps}$ given by
\begin{equation*}
D_{ps}:={\left( \begin{array}{cc}
(k_s/k_p)I_{52\times 52}& 0\\
0 & I_{52\times 52}
\end{array}
\right ).}
\end{equation*}
Then the regularized solution is given by
\begin{equation*}
\vec{\boldsymbol g}^\alpha_{\boldsymbol z}= \big(B_{\boldsymbol z}^\ast B_{\boldsymbol z}+\alpha I\big)^{-1}B_{\boldsymbol z}^\ast {\boldsymbol f}_F.
\end{equation*}
We plot the contours for the indicator function
\begin{equation}\label{Iz}
 I_{\boldsymbol z} = \frac{\|\vec{\boldsymbol g}^\alpha_{\boldsymbol z}\|_{l^2}}{\max\limits_{{\boldsymbol z}\in T}\|\vec{\boldsymbol g}^\alpha_{\boldsymbol z}\|_{l^2}}
\end{equation}
for all the sampling points ${\boldsymbol z}\in T$.

\autoref{fig5} shows the contour plots of $I_{\boldsymbol z}$ for the rigid pear and the reconstruction result.
\autoref{fig6} and \autoref{fig7} show the reconstructions for the cavity peanut and impedance kite, respectively.

Again, we use the multilevel ESM starting with a large sampling disk ($R=2.4$).
For the rigid pear, the cavity peanut and impedance kite ($\sigma=2$), the radius of the sampling disks are all found to be $R=0.6$. \autoref{fig8} shows the reconstructions of the multilevel ESM for the pear, the peanut, and the kite.

\begin{figure}[h!]
\begin{minipage}[t]{\linewidth}
\begin{minipage}[t]{0.40\linewidth}
\includegraphics[angle=0, width=\textwidth]{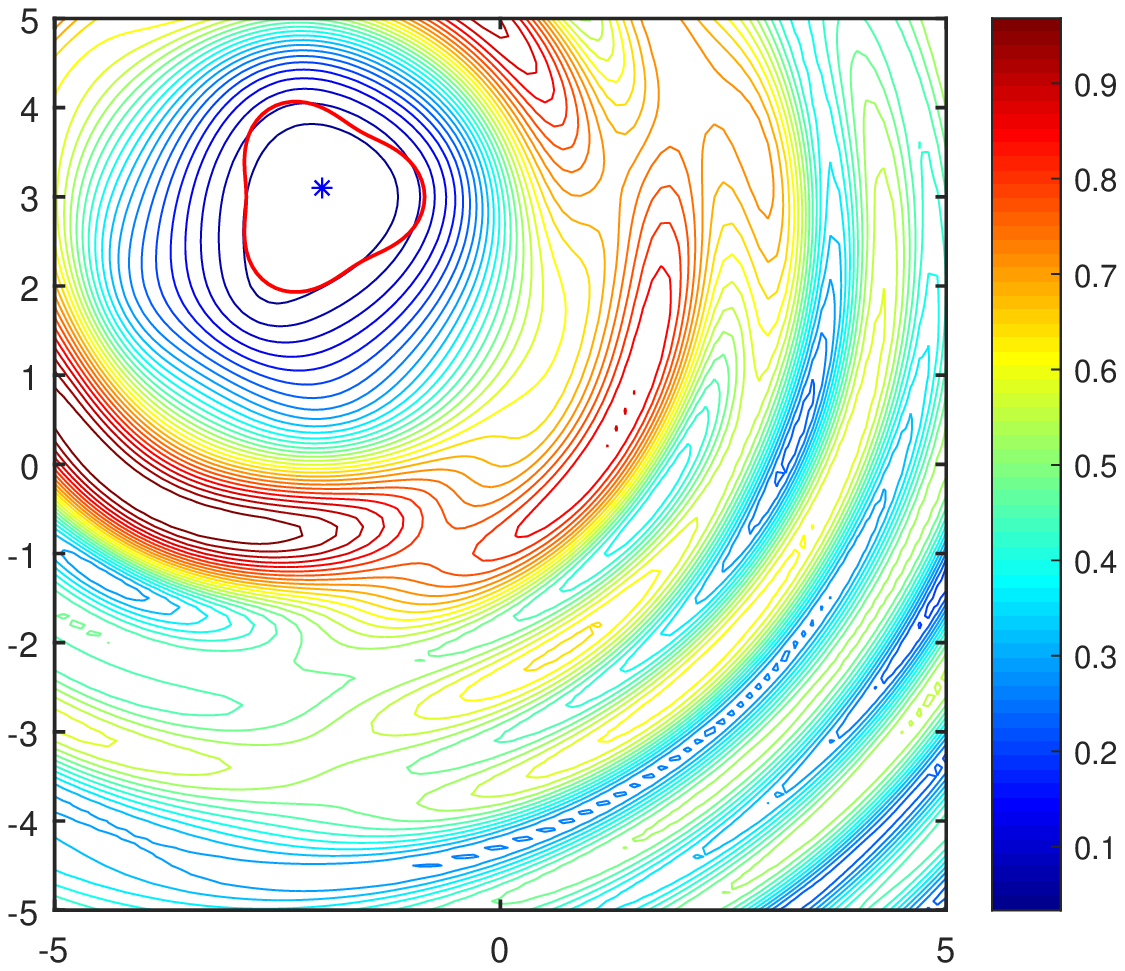}
\vspace{-0.9cm}

\end{minipage}
\hspace{0.4cm}
\begin{minipage}[t]{0.40\linewidth}
\includegraphics[angle=0, width=0.83\textwidth]{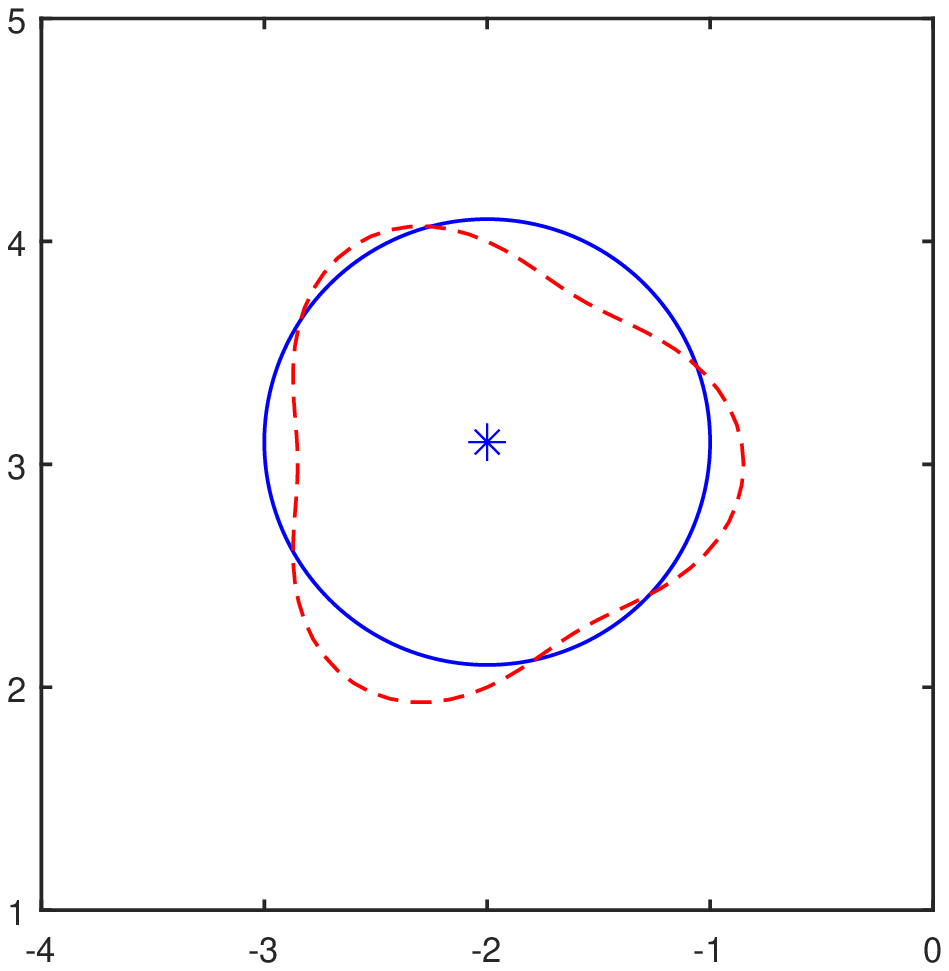}
\vspace{-0.4cm}

\end{minipage}
\end{minipage}
\caption{\label{fig5} Reconstructions of the rigid pear using the far field pattern: Left: contour plot of $I_z$; Right: reconstruction.}
\end{figure}

\begin{figure}[h!]
\begin{minipage}[t]{\linewidth}
\begin{minipage}[t]{0.40\linewidth}
\includegraphics[angle=0, width=\textwidth]{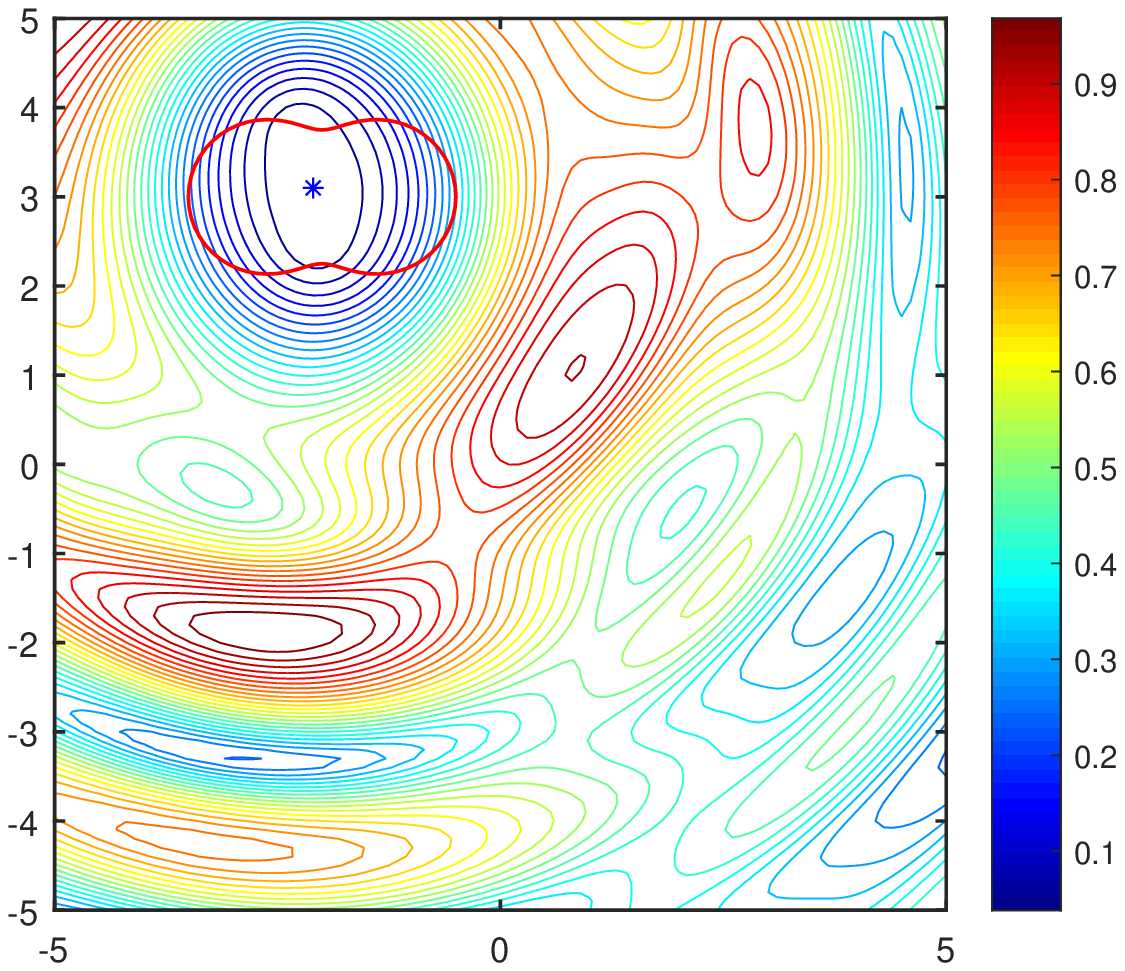}
\vspace{-0.9cm}

\end{minipage}
\hspace{0.4cm}
\begin{minipage}[t]{0.40\linewidth}
\includegraphics[angle=0, width=0.83\textwidth]{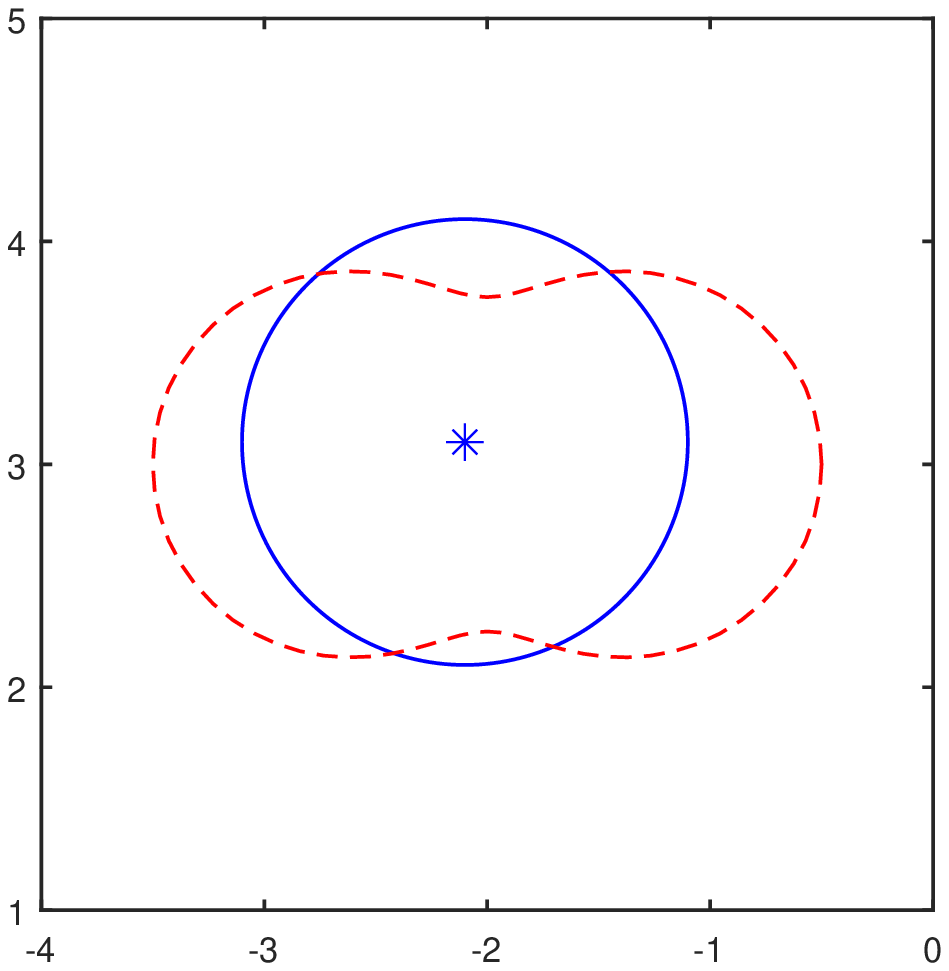}
\vspace{-0.4cm}

\end{minipage}
\end{minipage}
\caption{\label{fig6} Reconstructions of the cavity peanut using the far field pattern. Left: contour plot of $I_z$; Right: reconstruction.}
\end{figure}

\begin{figure}[h!]
\begin{minipage}[t]{\linewidth}
\begin{minipage}[t]{0.40\linewidth}
\includegraphics[angle=0, width=\textwidth]{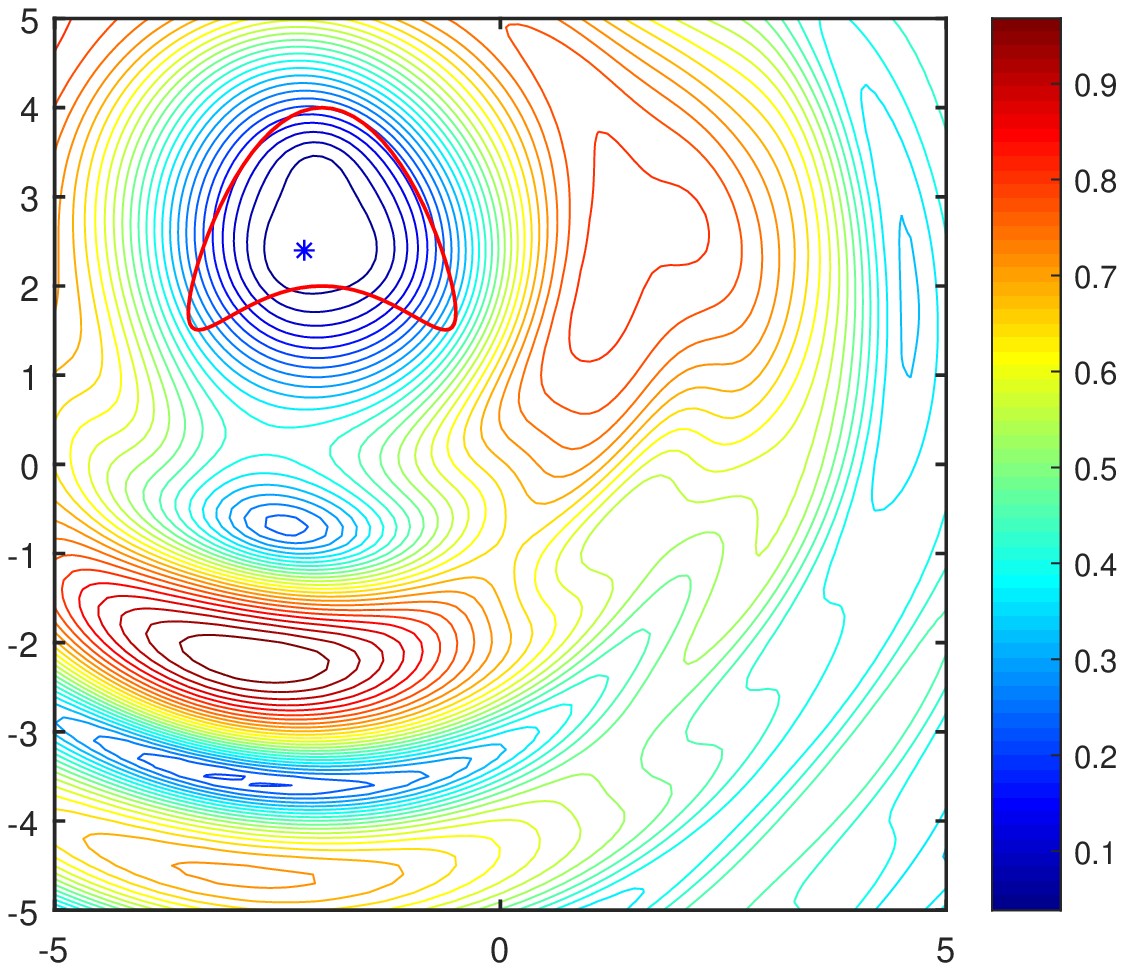}
\vspace{-0.9cm}

\end{minipage}
\hspace{0.4cm}
\begin{minipage}[t]{0.40\linewidth}
\includegraphics[angle=0, width=0.83\textwidth]{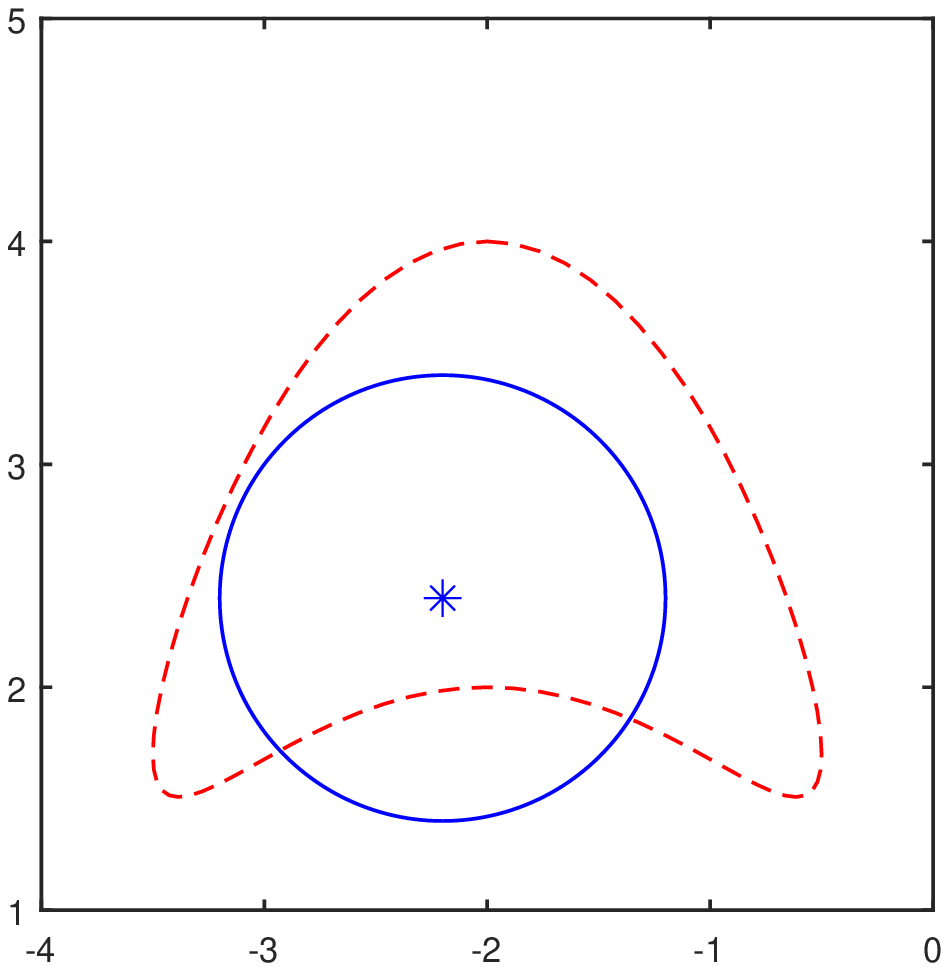}
\vspace{-0.4cm}

\end{minipage}
\end{minipage}
\caption{\label{fig7} Reconstructions of the impedance kite using the far field pattern. Left: contour plot of $I_z$; Right: reconstruction.}
\end{figure}

\begin{figure}[h!]
\begin{minipage}[t]{\linewidth}
\begin{minipage}[t]{0.30\linewidth}
\includegraphics[angle=0, width=\textwidth]{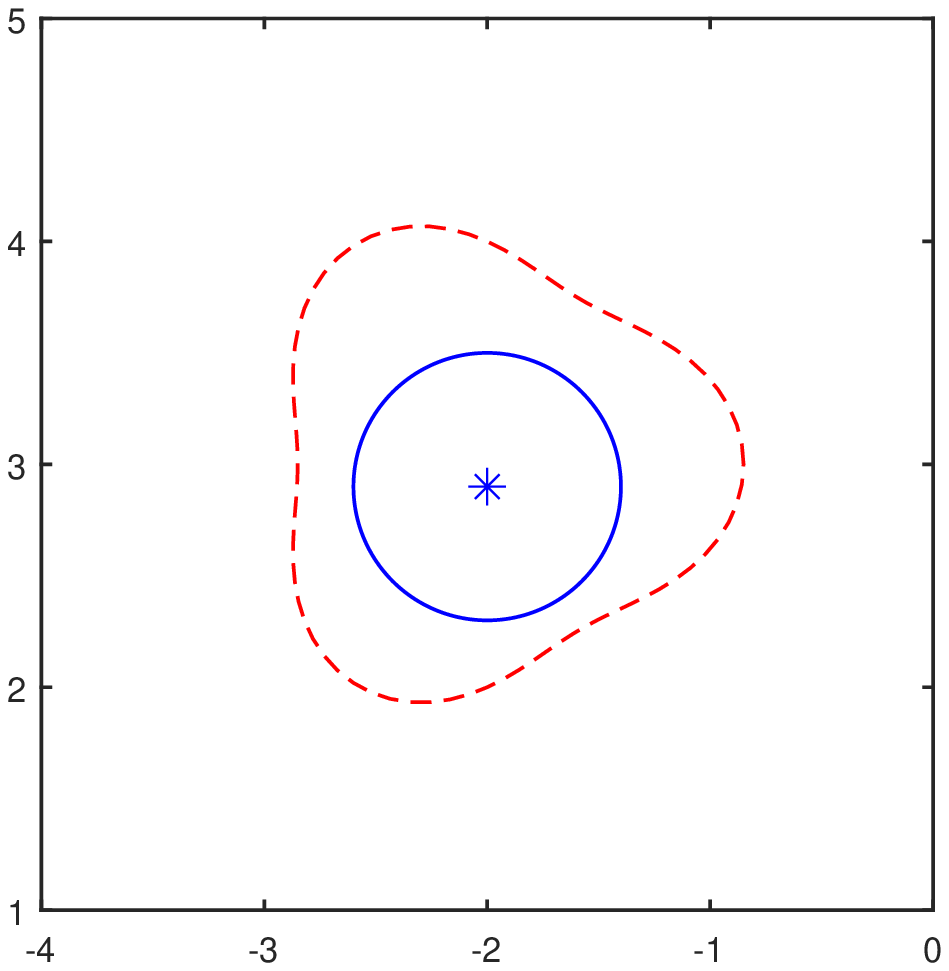}
\vspace{-0.9cm}

\vspace{0.1cm}
\end{minipage}
\hspace{0.2cm}
\begin{minipage}[t]{0.30\linewidth}
\includegraphics[angle=0, width=\textwidth]{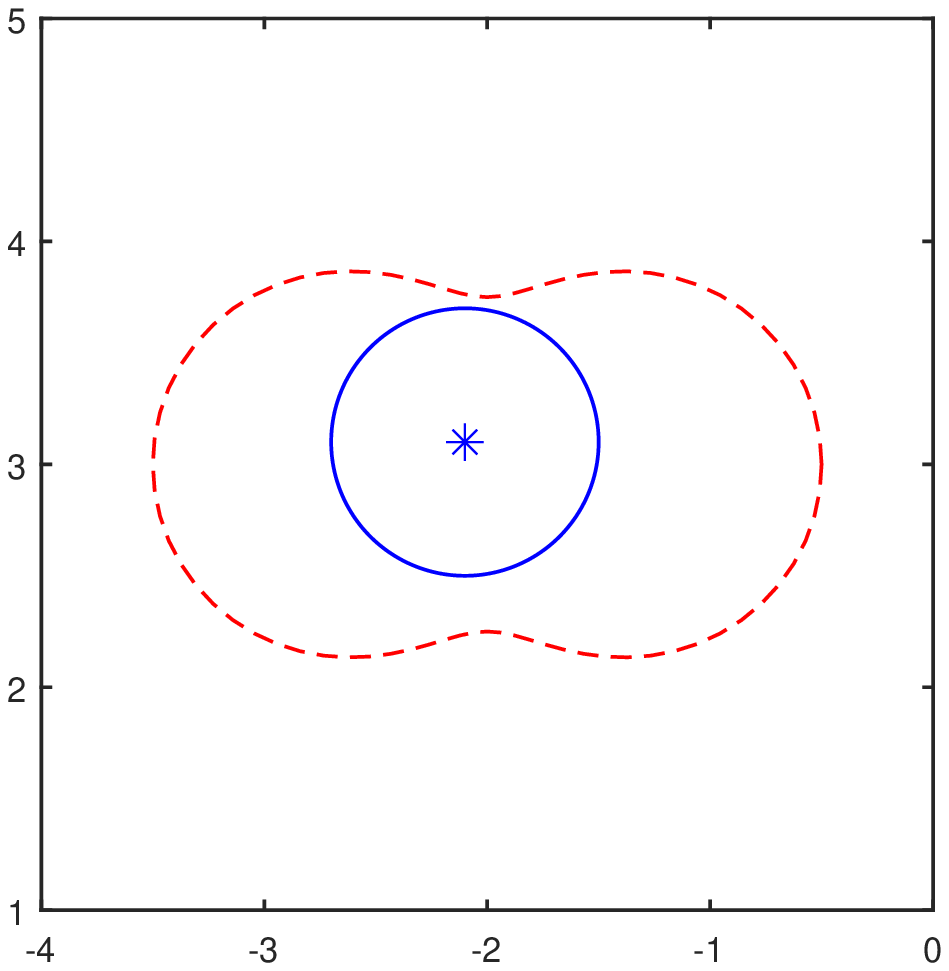}
\vspace{-0.9cm}

\vspace{0.1cm}
\end{minipage}
\hspace{0.2cm}
\begin{minipage}[t]{0.30\linewidth}
\includegraphics[angle=0, width=\textwidth]{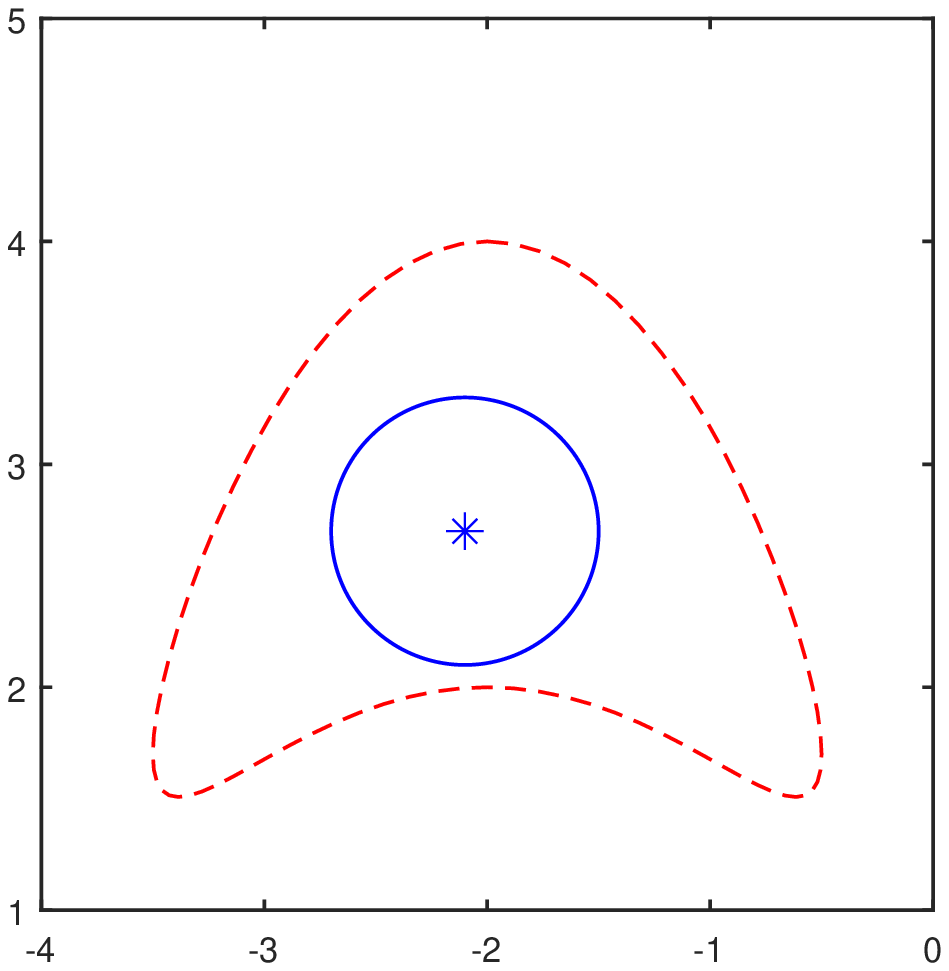}
\vspace{-0.9cm}

\vspace{0.1cm}
\end{minipage}
\end{minipage}
\caption{\label{fig8} The reconstruction results of the multilevel ESM. Left: the rigid pear obstacle; Middle: the cavity peanut obstacle; Right: the impedance kite obstacle.}
\end{figure}


\begin{thebibliography}{99}
\bibitem{AmmariCalmonIakovleva2008SIAMIS} H. Ammari, P. Calmon and E. Iakovleva,
	{\em Direct elastic imaging of a small inclusion.}
	SIAM J. Imaging Sci. {\bf1}, 169-187, 2008.
\bibitem{AlvesKress2002} C. Alves and R. Kress,
	{\em On the far field operator in elastic obstacle scattering.}
	IMA J. Appl. Math. {\bf67}, no.1, 1-21, 2002.
\bibitem{Arens2001} T. Arens,
	{\em Linear sampling methods for 2D inverse elastic wave scattering.}
	Inverse Problems {\bf17}, 1445-1464, 2001.
\bibitem{BaoHuSunYin2018JMPA} G. Bao, G. Hu, J. Sun and T. Yin,
	{\em Direct and inverse elastic scattering from anisotropic media.}
	 J. Math. Pures Appl. {\bf117}, 263-301, 2018.
\bibitem{BonnetConstantinescu2005IP} M. Bonnet and A. Constantinescu,
	{\em Inverse problems in elasticity.}
	Inverse Problems {\bf21}, R1-R50, 2005.	
\bibitem{Charalambopoulos2007} A. Charalambopoulos, A. Kirsch, K. Anagnostopoulos, D. Gintides and K. Kiriaki,
	{\em The factorization method in inverse elastic scattering from penetrable bodies.}
	Inverse Problems {\bf23}, 27-51, 2007.
\bibitem{ChenHuang2015SCM} Z. Chen and G. Huang,
	{\em Reverse time migration for extended obstacles: elastic waves.}
	Sci. China Math. {\bf45}, no.8, 1103-1114, 2015.
\bibitem{ColtonKirsch1996IP} D. Colton and A. Kirsch,
	{\em A simple method for solving inverse scattering problems in the resonance region.}
	{Inverse Problems \bf12}, 383-393, 1996.
\bibitem{ColtonHaddar2005IP} D. Colton and H. Haddar,
	{\em An application of the reciprocity gap functional to inverse scattering theory.} 
	Inverse Problems {\bf21}, 383-398, 2005.
\bibitem{ColtonKress2013} D. Colton and R. Kress, 
	{\em Inverse Acoustic and Electromagnetic Scattering Theory} (3rd ed.), Springer, 2013.
\bibitem{DiCristoSun2006IP} M. Di Cristo and J. Sun,
	{\em An inverse scattering problem for a partially coated buried obstacle.}
	Inverse Problems {\bf22}, no. 6, 2331-2350, 2006.
\bibitem{HahnerHsiao1993IP} P. H\"{a}hner and G. Hsiao,
	{\em Uniqueness theorems in inverse obstacle scattering of elastic waves.}
	Inverse Problems {\bf9}, 525-534, 1993.
\bibitem{HuKirschSini2013} G. Hu, A. Kirsch and M. Sini,
	{\em Some inverse problems arising from elastic scattering by rigid obstacles.}
	Inverse Problems {\bf29}, 015009, 2013.
\bibitem{HuLiLiuSun2014SIAMIS} G. Hu, J. Li, H. Liu and H. Sun,
	{\em Inverse elastic scattering for multiscale rigid bodies with a single far-field pattern.}
	SIAM J. Imaging Sci. {\bf7}, 1799-1825, 2014.
\bibitem{ItoJinZou2012IP} K. Ito, B. Jin and J. Zou,
	{\em A direct sampling method to an inverse medium scattering problem.} {Inverse Problems \bf28}, 025003, 2012.
\bibitem{JiLiu} X. Ji and X. Liu,
	{\em Inverse elastic scattering problems with phaseless far field data.}
	arxiv:1812.02359v2, 2018.
\bibitem{JiLiuXi2018} X. Ji, X. Liu and Y. Xi,
	{\em Direct sampling methods for inverse elastic scattering problems.}
	Inverse Problems {\bf34}, 035008, 2018.
\bibitem{LiWangWangZhao2016IP} P. Li, Y. Wang, Z. Wang and Y. Zhao,
	{\em Inverse obstacle scattering for elastic waves.}
	Inverse Problems {\bf32}, 115018, 2016.
\bibitem{LiuSun2018} J. Liu and J. Sun, 
	{\em Extended sampling method in inverse scattering.} 
	Inverse Problems {\bf34}, 085007, 2018.
\bibitem{MonkSun2007IPI} P. Monk and J. Sun,
	{\em Inverse scattering using finite elements and gap reciprocity.} 
	Inverse Probl. Imaging {\bf1}, no. 4, 643-660, 2007.
\bibitem{PotthastSylvestrKusiak2003IP} R. Potthast, J. Sylvester and S. Kusiak,
	{\em A `range test' for determining scatterers with unknown physical properties.} 
	Inverse Problems {\bf19}, no. 3, 533-547, 2003.
\bibitem{WangQiu2017} Z. Wang and S. Qiu, 
	{\em A numerical approximation of the two-dimensional elastic wave scattering problem via integral equation method.} 
	 Appl. Numer. Math. {\bf113}, 156-167, 2017.

\end{thebibliography}
\end{document}